\documentclass[11pt]{article} 

\usepackage[utf8]{inputenc} 

\usepackage{geometry} 
\usepackage[hidelinks]{hyperref}
\usepackage{authblk}
\usepackage{graphicx} 
\usepackage{amsfonts,amsthm,amsmath,amssymb,esint}
\usepackage{slashed}
\usepackage{tikz}
\usetikzlibrary{arrows}
\usepackage{stmaryrd}
\newcommand{\R}{\mathbb{R}}
\newcommand{\C}{\mathbb{C}}
\newcommand{\Z}{\mathbb{Z}}
\newcommand{\Q}{\mathbb{Q}}

\newcommand{\N}{\mathbb{N}}

\newcommand{\F}{\mathbb{F}}
\newcommand{\Pb}{\mathbb{P}}

\newcommand{\M}{\mathbb{M}}

\newcommand{\Li}{\mathcal{L}}

\newcommand{\Ai}{\mathcal{A}}

\newcommand{\Pis}{\mathcal{P}}

\newcommand{\Ci}{\mathcal{C}}

\newcommand{\Zi}{\mathcal{Z}}
\newcommand{\Ri}{\mathcal{R}}
\newcommand{\Di}{\mathcal{D}}
\newcommand{\Ei}{\mathcal{E}}
\newcommand{\Fi}{\mathcal{F}}
\newcommand{\Oi}{\mathcal{O}}

\newcommand{\GU}{\mathfrak{U}}

\newcommand{\Gl}{\mathfrak{l}}
\newcommand{\Gg}{\mathfrak{g}}

\newcommand{\id}{\operatorname{id}}

\newcommand{\pd}{\partial}

\newcommand{\tr}{\operatorname{tr}}
\newcommand{\End}{\operatorname{End}}

\newcommand{\longto}{\longrightarrow}

\newcommand{\Un}{\operatorname{U}}

\newcommand{\Mn}{\operatorname{M}}
\newcommand{\GeL}{\operatorname{GL}}

\newcommand{\Spec}{\operatorname{Spec}}

\newcommand{\Ran}{\operatorname{Ran}}

\newtheorem{thm}{Theorem}[section]
\newtheorem*{thm*}{Theorem}
\newtheorem{Lemma}[thm]{Lemma}
\newtheorem{prop}[thm]{Proposition}
\newtheorem{cor}[thm]{Corollary}

\numberwithin{equation}{section}

\theoremstyle{definition}
\newtheorem{dfn}[thm]{Definition}

\newtheorem*{HdgCon}{Hodge's conjecture}

\newtheorem{Remark}[thm]{Remark}

\newtheorem*{Ack}{Acknowledgment}

\newcommand{\Ad}{\operatorname{Ad}}

\newcommand{\bone}{\mathbf{1}}

\newcommand{\Ker}{\operatorname{Ker}}

\newcommand{\ch}{\operatorname{ch}}
\newcommand{\bwedge}{\mbox{\Large $\wedge$}}

\newcommand{\cs}{\operatorname{cs}}

\usepackage{booktabs} 
\usepackage{array} 
\usepackage{paralist} 
\usepackage{verbatim} 
\usepackage{subfig} 

\usepackage{fancyhdr} 
\evensidemargin=\oddsidemargin
\usepackage{sectsty}
\allsectionsfont{\sffamily\mdseries\upshape} 

\usepackage[nottoc,notlof,notlot]{tocbibind} 
\usepackage[titles,subfigure]{tocloft} 

\title{Hodge classes of Chern character forms on compact Kähler manifolds}
\author{Andreas Andersson}
\affil{\footnotesize Chalmers University of Technology, Mathematical Sciences, Maskingränd 2, 412 58 Gothenburg, Sweden
\\University of Oslo, Department of Mathematics, P.O. Box 1053 Blindern, 0316 Oslo, Norway}
\affil{\footnotesize Mathematics Subject Classification 2010 -- Primary: 14C30
; Secondary: 32L10
, 19L10 
}

\begin{document}
\maketitle
\abstract

In this paper we show that every rational cohomology class of type $(p,p)$ on a compact Kähler manifold can be representated as a differential $(p,p)$-form given by an explicit formula involving a \v{C}ech cocycle.
First we represent Chern characters of smooth vector bundles by \v{C}ech cocycles with values in the sheaf of differential forms. We then consider the behavior of these cocycles with respect to the Hodge structure on cohomology when the base manifold is compact Kähler.  

\tableofcontents 

\section*{Introduction}

This paper is a contribution to our understanding of the following claim about cohomology classes on smooth projectve varieties (over $\C$) (see 
\cite{Deli1, Lewis1, Vois11} for background and motivation):
\begin{HdgCon}[{\cite{Hodg2}}]
Let $\M$ be a smooth projective variety and let $p\in\{1,\dots,\dim_\C\M\}$. Then every class in $H^{p,p}(\M)\cap H^{2p}(\M;\Q)$ is a rational linear combination of the classes of complex-analytic subvarieties of $\M$.
\end{HdgCon}
In order to understand the usefulness of the present work, note that the Hodge conjecture has a nice reformulation in terms of vector bundles over $\M$. To see that, define the $p$th Chern character of a smooth vector bundle $\Ei$ on any smooth manifold $\M$ to be the degree-$2p$ component $\ch_p(\Ei)$ of the Chern character $\ch_*(\Ei)$, 
$$
\ch_*(\Ei)=\sum_{p=0}^{\dim_\C\M}\ch_p(\Ei)\in\bigoplus_{p=0}^{\dim_\C\M}H^{2p}(\M;\Q).
$$
If $\Theta$ is the curvature 2-form of any given connection $\theta$ on $\Ei$, a de Rham representative of $\ch_p(\Ei)$ is given by the closed differential $2p$-form $\ch_p(\theta)=\tr(\Theta^{\wedge p})/p!(-2\pi i)^p$. The induced map on topological $K$-theory,
$$
\ch_p:K^0(\M)\to H^{2p}(\M;\Q)
$$
is a surjective group homomorphism. That is, every rational cohomology class is a $\Q$-linear combination of $\ch_p(\Ei)$'s. 
For an arbitrary smooth vector bundle $\Ei$ on a compact complex-analytic manifold $\M$, the differential $2p$-form $\ch_p(\theta)$  splits in components of all possible types $(2p,0),(2p-1,1),\dots,(0,2p)$ with respect to the decomposition on cohomology induced by the given complex-analytic structure. Holomorphic vector bundles have Chern characters of type $(p,p)$, since on such bundles there are connections whose curvature 2-form is of type $(1,1)$. 

As explained in \cite[\S1, p.1]{Vois2}, if $\M$ is a smooth projective variety then the Chern classes of holomorphic vector bundles span the same subspace of $H^{2p}(\M;\Q)$ as do the complex-analytic subvarieties of $\M$. Since the set of Chern classes (elementary symmetric polynomials) and the set of Chern characters (power sums) both generate the algebra of symmetric polynomials with rational coefficients, the Chern characters generate the same classes as the Chern classes when the coefficients are in $\Q$. 

With these facts in mind, we can phrase the Hodge conjecture in the following equivalent way:
\begin{HdgCon}
Let $\M$ be a smooth projective variety and let $p\in\{1,\dots,\dim_\C\M\}$. Then every class in $H^{p,p}(\M)\cap H^{2p}(\M;\Q)$ is a $\Q$-linear combination of the $p$th Chern characters of holomorphic vector bundles over $\M$.
\end{HdgCon}
Let us now explain the result of the present paper, which splits into two. 

We first show that the de Rham cohomology class $\ch_p(\Ei)\in H^{2p}_{\rm DR}(\M;\R)$ can be represented, under the \v{C}ech--de Rham isomorphism with respect to any finite covering of $\M$ by open subsets for which each intersection is contractible, by a \v{C}ech $p$-cocycle $\xi^{E,p}$ with values in the sheaf $\Ai^p$ of differential $p$-forms on $\M$. To see this we shall use Swan's correspondence between smooth vector bundles and projective $C^\infty(\M)$-modules and calculate the Chern character form $\ch_p(I^E)$ for a choice of idempotent $I^E$ defining the module $\Gamma^\infty_0(\M;\Ei)$ of smooth sections of $\Ei$ vanishing at infinity. 
Let $\GU=(U_\alpha)_{\alpha=1,\dots,n}$ be a finite open covering of $\M$ which trivializes the vector bundle $\Ei$. The idempotent $I^E$ is constructed directly from the transition functions $g_{\alpha\beta}:U_\alpha\cap U_\beta\to\GeL(r,\C)$ of $\Ei$ and a partition of unity $(\rho_\alpha)_{\alpha=1,\dots,n}$ subordinate to the covering, as in \cite[Thm. 1.1.4]{Fedo1}. 
Let $(\theta^E_\alpha)_{\alpha=1,\dots,n}$ be the local connection 1-forms of the ``Levi-Civita connection" on $\Ei$ defined by the idempotent $I^E$ (see the main text). 
A straightforward calculation shows that the $2p$-form
\begin{align*}
\widetilde{\ch}_p(\theta^E):=&\sum_{\alpha_0=1}^n\rho_{\alpha_0}\tr_{\C^r}\big((d\theta_{\alpha_0}^E)^{\wedge p}\big)
\end{align*}
is given by
\begin{align*}
(-1)^{(p-1)!}\sum_{\alpha_0,\dots,\alpha_p=1}^n\rho_{\alpha_0}\tr_{\C^r}(d\log g_{\alpha_0\alpha_1}\wedge d\log g_{\alpha_1\alpha_2}\wedge\cdots\wedge d\log g_{\alpha_{p-1}\alpha_p})\wedge d\rho_{\alpha_1}\wedge\cdots\wedge d\rho_{\alpha_p},
\end{align*}
where $d\log g_{\alpha\beta}:=dg_{\alpha\beta}\, g_{\alpha\beta}^{-1}$. For comparison, note that the ordinary $p$th Chern character form of $\theta^E$ is
\begin{align*}
\ch_p(\theta^E)=&\sum_{\alpha_0=1}^n\rho_{\alpha_0}
\tr_{\C^r}\big((d\theta_{\alpha_0}^E+\theta_{\alpha_0}^E\wedge\theta_{\alpha_0}^E)^{\wedge p}\big).
\end{align*}
By definition, $\ch_p(\Ei)=[\ch_p(\theta^E)]$. The first purpose of this note is to show that, quite suprisingly, also the modified $p$th Chern character form $\widetilde{\ch}_p(\theta^E)$ is a de Rham representative of the $p$th Chern character of $\Ei$:
$$
\ch_p(\Ei)=[\widetilde{\ch}_p(\theta^E)]\in H_{\rm DR}^{2p}(\M;\C).
$$ 
In this way one enters, we shall see, the setting of Weil's proof of the \v{C}ech--de Rham isomorphism. Indeed, the above expression for $\widetilde{\ch}_p(\theta^E)$ shows that $\ch_p(\Ei)$ is represented in de Rham cohomology by the image under Weil's isomorphism of the $p$-cocycle $\xi^{E,p}$ with values in the sheaf $\Ai^p$ of differential $p$-forms on $\M$ defined by
$$
\xi^p_{\alpha_0\cdots\alpha_p}:=(-1)^{(p-1)!}\tr_{\C^r}(d\log g_{\alpha_0\alpha_1}\wedge d\log g_{\alpha_1\alpha_2}\wedge\cdots\wedge d\log g_{\alpha_{p-1}\alpha_p}).
$$
We call $\xi^{E,p}$ the ``smooth Atiyah $p$-cocycle" of $\Ei$, by analogy with the familiar Atiyah cocycles \cite{Atiy1} of holomorphic vector bundles on complex-analytic manifolds. In contrast to their holomorphic counterparts, the smooth Atiyah cocycles do not represent an obstruction to the existence of a connection, since every short exact sequence of smooth vector bundles has a smooth splitting. Nevertheless, we will show here that the smooth Atiyah cocycles are very useful when dealing with Chern characters of smooth vector bundles. To mention related previous work, Bott constructed a representative of the Chern classes $c_p(\Ei)$ in the total \v{C}ech--de Rham complex \cite[\S5, p.30]{Bott2}, while Brylinski--McLaughlin gave a $2p$-cocycle representative of $c_p(\Ei)$ with values in the constant sheaf $\Z_\M$ \cite{BrML1}.  

As an application of our new representative for the Chern character we will characterize exact $2p$-forms on $\M$ as modified Chern characters $\widetilde{\ch}_p(\theta^E)$ of Levi-Civita connections of trivial smooth vector bundles. Since every integer de Rham cohomology class has a representative of the type $p!\,\widetilde{\ch}_p(\theta^E)$ if the covering consists of contractible sets, in this case we obtain that every closed $2p$-form on $\M$ with integral periods is given by $p!\,\widetilde{\ch}_p(\theta^E)$ for some smooth vector bundle $\Ei$. The same partition of unity can be used for all such $2p$-forms.

With a \v{C}ech-cocycle representative for each Chern character $\ch_p(\Ei)$ available, we deduce that all $d$-closed $\Ai^p$-valued $p$-cocycles are $\Q$-linear combinations of the above-mentioned cocycles $\xi^{E,p}$ for smooth vector bundles $\Ei$. 

Next we assume that $\M$ has a complex structure, and use the fact \cite{Harv1} that the Dolbeault isomorphism $\Di^{p,p}:\check{H}^p(\GU;\Omega^p)\to H^{p,p}_{\bar{\pd}}(\M)$ can be constructed in a very similar way to Weil's \v{C}ech--de Rham isomorphism. We define $\xi^{E,p,0}$ to be the $\Ai^{p,0}$-valued part of $\xi^{E,p}$, regardless of whether $\Ei$ has a holomorphic structure or not. Then $\xi^{E,p,0}$ takes values in the sheaf $\Omega^p\subset\Ai^{p,0}$ of holomorphic $(p,0)$-forms if and only if $\bar{\pd}\xi^{E,p,0}=0$. We shall show that, when $\M$ has furthermore a Kähler structure, the Atiyah $(p,0)$-cocycles of some smooth vector bundles give rise to $\bar{\pd}$-closed differential $(p,p)$-forms which represent all $(p,p)$-classes:
\begin{thm*}\label{steponethmintro}
Let $\M$ be a compact Kähler manifold and let $\GU=(U_\alpha)_{\alpha=1,\dots,n}$ be a Stein covering of $\M$. Then every class in $H^{p,p}(\M)\cap H^{2p}(\M;\Q(p))$ has a representative which is a $\Q$-linear combination of $(p,p)$-forms of the kind 
$$
\Di^{p,p}(\xi^{p,0})=(-1)^{(p-1)!}\sum_{\alpha_0,\dots,\alpha_p}\rho_{\alpha_0}\tr_{\C^r}(\pd\log g_{\alpha_0\alpha_1}\wedge\cdots\wedge\pd\log g_{\alpha_{p-1}\alpha_p})\wedge \bar{\pd}\rho_{\alpha_1}\wedge\cdots\wedge\bar{\pd}\rho_{\alpha_p}
$$ 
for $\Ai^{p,0}$-valued $p$-cocycles $\xi^{p,0}=\xi^{E,p,0}$ of smooth vector bundles $\Ei$ on $\M$ for which $\xi^p=\xi^{p,0}$ (in particular $\bar{\pd}\xi^{p,0}=0$). 
\end{thm*}
This result could be helpful in finding the solution of the Hodge conjecture. Indeed, before we only knew that
\begin{enumerate}[(i)]
\item{if $\Ei$ is holomorphic then for $p=1$ we have that $\xi^{E,1,0}$ corresponds to $c_1(\Ei)=\ch_1(\Ei)$ under the Dolbeault isomorphism (apply the fiber trace to \cite[Prop. 4]{Atiy1}),}
\item{if $\Ei$ is holomorphic and $\Theta^{1,1}$ denotes the $(1,1)$-part of the curvature of a connection on $\Ei$ with $\Theta^{0,2}=0$ then $(-2\pi i)^{-p}\tr(\Theta^{1,1})^{\wedge p}$ is a Dolbeault representative of $p!\ch_p(\Ei)$ \cite[\S\S1.2, 1.3]{ABST1},}
\item{every $(p,p)$-class can be represented by a differential $(p,p)$-form of the type
$$
\Di^{p,p}(\eta):=\sum_{\alpha_0,\dots,\alpha_p}\rho_{\alpha_0}\eta_{\alpha_0\cdots\alpha_p}\wedge \bar{\pd}\rho_{\alpha_1}\wedge\cdots\wedge\bar{\pd}\rho_{\alpha_p}
$$
for some $\Omega^p$-valued $p$-cocycle $\eta$ (this is the result of Harvey referred to above and recalled below in \S\ref{Harveysec}), and}
\item{the Chern character forms $\widetilde{\ch}_p(\theta^E)$ give rise to all $(p,p)$-classes (by \S\ref{smoothsection} of this paper), but possibly all $(p,q)$-components of $\xi^{E,p}$ are needed to achieve this.}
\end{enumerate}
The second and final step needed for proving the Hodge conjecture would be to show that, when $\M$ is projective, the Atiyah classes $\xi^{E,p,0}$ appearing in the above theorem can always be chosen to come from vector bundles admitting holomorphic structures. 
It might be that the second step could be taken by showing that the group $K^0(\M)$ is generated by classes of smooth vector bundles admitting connections which are special solutions to the Yang--Mills equation (viz. the ``$\Omega$-anti-selfdual instanton'' solutions of \cite[Prop. 1.2.2]{Tian5}). Namely, if $\Ei$ is a smooth vector bundle defined by an idempotent $I^E$ (as in \S\ref{smoothsection}) whose Levi-Civita connection is such a special solution to the Yang--Mills equation then the condition $\xi^{E,p}=\xi^{E,p,0}$ seems to gaurantee that $\Ei$ is holomorphic. These Yang--Mills connections seem to appear in quantization, and the desired property of $K^0(\M)$ can possibly be deduced from studying an algebra of Toeplitz operators. This is speculative but worth investigating.


In \S\ref{algsection} we make a simple strengthening of the main result in the case of projective Kähler manifolds. We formulate the result explicitly in the language of algebraic \v{C}ech cohomology, just to make the results of the paper more readily applicable in that framework. 

\begin{Ack}
I thank Sergey Neshveyev and Adam Rennie for comments on the first half of the paper. This research was supported by ERC (grant 307663-NCGQG).
\end{Ack}

\section*{Notation}
If $\M$ is a smooth manifold we let $\Ci^\infty_\M$ be its structure sheaf, i.e. the sheaf of smooth $\C$-valued functions on $\M$. We also use $C^\infty(U)=\Ci^\infty_\M(U)$ to denote the algebra of smooth functions on an open set $U\subseteq\M$. In this paper, by a ``smooth vector bundle" on $\M$ we will mean a locally free $\Ci^\infty_\M$-module of finite rank. We denote by $\Z(p):=(2\pi i)^p\Z$ be the subgroup of $\C$ generated by $(2\pi i)^p$.
If $\M$ has a complex-analytic structure, we denote by $\Oi_\M$ the sheaf of holomorphic functions on $\M$, and a ``holomorphic vector bundle" on $\M$ is a smooth vector bundle with a specified choice of transition functions which are holomorphic. We write $\Ai^p$ and $\Ai^{p,q}$ for the $\Ci^\infty_\M$-modules of smooth $p$- and $(p,q)$-forms on $\M$, respectively, and by $\Omega^p$ the $\Oi_\M$-module of $\bar{\pd}$-closed $(p,0)$-forms, where $d=\pd+\bar{\pd}$ is the splitting of the de Rham differential (exterior derivative) into $(1,0)$- and $(0,1)$-parts.

\section{\v{C}ech cocycles for Chern characters}\label{smoothsection}


\subsection{The map $\Di\Ri$}
Let $\M$ be a smooth manifold and let $\GU=(U_\alpha)_{\alpha=1,\dots,n}$ be a finite open covering of $\M$ (finiteness of the covering will be needed already from Lemma \ref{Fedosovlemma} and onwards so we assume it throughout). We write 
$$
U_{\alpha_0\cdots\alpha_q}:=U_{\alpha_0}\cap\cdots\cap U_{\alpha_q}
$$
for all $\alpha_0,\dots,\alpha_q\in\{1,\dots,n\}$. For each $q\geq 0$ and each sheaf $\Fi$ (of complex vector spaces) we have the space $\check{C}^q(\GU;\Fi)$ of $\Fi$-valued \v{C}ech $q$-cochains on $\M$ with respect to $\GU$. Let $\delta:\check{C}^q(\GU;\Fi)\to\check{C}^{q+1}(\GU;\Fi)$ be the \v{C}ech coboundary operator, and let $\check{Z}^q(\GU;\Fi)$ and $\check{H}^q(\GU;\Fi)$ be the spaces of $q$-cocycles and cohomology classes of $q$-cocycles respectively. We write $q$-cochains in the form 
$$
\eta=(\eta_{\alpha_0\cdots\alpha_q})_{\alpha_0,\cdots,\alpha_q=1,\dots,n}\in\check{C}^q(\GU;\Fi)\subset\bigoplus_{\alpha_0,\cdots,\alpha_q=1}^n\Fi(U_{\alpha_0\cdots\alpha_q})
$$ 
where the component $\eta_{\alpha_0\cdots\alpha_q}$ belongs to the space $\Fi(U_{\alpha_0\cdots\alpha_q})$ of sections of $\Fi$ over $U_{\alpha_0\cdots\alpha_q}$. Here only strictly increasing $(q+1)$-tuples $(\alpha_0,\dots,\alpha_q)$ are allowed, but we require the cocycles to be antisymmetric under exchange of any two indices $\alpha,\beta\in\{1,\dots,n\}$ so one could restrict attention to orded $(q+1)$-tuples. 

Now take $\Fi$ to be the $\Ci^\infty_\M$-module $\Ai^p$. The de Rham differential $d:\Ai^p\to\Ai^{p+1}$ gives rise to the differential 
$$
d:\check{C}^q(\GU;\Ai^p)\to\check{C}^q(\GU;\Ai^{p+1}),\qquad d\eta:= (d\eta_{\alpha_0\cdots\alpha_q})_{\alpha_0,\dots,\alpha_q=1,\dots,n},
$$
and $d$ commutes with the \v{C}ech coboundary operator $\delta$. Using the double complex $(\check{C}^q(\GU;\Ai^p),d,\delta)_{q,p\geq 0}$, Weil constructed an explicit isomorphism between the \v{C}ech cohomology $\check{H}^*(\GU;\C_\M)$ of the constant sheaf $\C_\M$ and the de Rham cohomology $H_{\rm DR}^*(\M;\C)$ of $\M$ \cite{Weil1} (see also \cite[Thm. 8.9]{BoTu1}). Here we will construct cocycle representatives of Chern characters in the Chern--Weil picture and relate it to Weil's approach to the \v{C}ech--de Rham isomorphism. 

The \textbf{cup product} of two cochains $\xi\in\check{C}^q(\GU;\Ai^p)$ and $\eta\in\check{C}^{q'}(\GU;\Ai^{p'})$ is the cochain $\xi\cup\eta\in\check{C}^{q+q'}(\GU;\Ai^{p+p'})$ defined by \cite[p. 174]{BoTu1}
\begin{equation}\label{cupprod}
(\xi\cup\eta)_{\alpha_0\cdots\alpha_{q+q'}}:=(-1)^{pq'}\xi_{\alpha_0\cdots\alpha_q}\wedge\eta_{\alpha_q\cdots\alpha_{q+q'}}.
\end{equation}

\begin{Lemma}\label{DRlemma}
Define a map $\Di\Ri^{p,q}:\check{C}^q(\GU;\Ai^p)\to\Ai^{p+q}(\M)$ by
$$
\Di\Ri^{p,q}(\xi):=\sum_{\alpha_0,\dots,\alpha_q}\rho_{\alpha_0}\xi_{\alpha_0\cdots\alpha_q}\wedge d\rho_{\alpha_1}\wedge\cdots\wedge d\rho_{\alpha_q}.
$$
Then $\Di\Ri^{p,q}$ maps
\begin{enumerate}[(i)]
\item{cup products of cocycles into wedge products of differential forms, and}
\item{$d$-closed $\delta$-cocycles into closed differential forms. In fact,
\begin{equation}\label{dofDReq}
d\Di\Ri^{p,q}(\xi)=(-1)^{q+1}\Di\Ri^{p,q+1}(\delta\xi)+(-1)^q\Di\Ri^{p+1,q}(d\xi).
\end{equation}}
\end{enumerate}
\end{Lemma}
\begin{proof}
For (i) we can follow \cite[Remark IV.8.10]{Dema1} (the only difference is that we use $\Ai^p$-valued cocycles, not $\R_\M$-valued, so the sign $(-1)^{pq'}$ has to be inserted). In more detail, for $\xi$ with $\delta\xi=0$, the restriction of $\Di\Ri^{p,q}(\eta)$ to the open subset $U_{\alpha_q}$ is given by
\begin{align*}
\Di\Ri^{p,q}(\eta)|_{U_{\alpha_q}}&=\sum_{\alpha_{q+1},\dots,\alpha_{q+q'}}\eta_{\alpha_q\cdots\alpha_{q+q'}}\wedge d\rho_{\alpha_{q+1}}\wedge\cdots\wedge d\rho_{\alpha_{q+q'}}
\end{align*}
(see \cite[Remark 8.59]{RoSa1}). So for $\xi$ and $\eta$ as in \eqref{dofDReq}, we have
\begin{align*}
\Di\Ri^{p,q}(\xi\cup\eta)&=(-1)^{pq'}\sum_{\alpha_0,\dots,\alpha_{q+q'}}\rho_{\alpha_0}\xi_{\alpha_0\cdots\alpha_q}\wedge\eta_{\alpha_q\cdots\alpha_{q+q'}}\wedge d\rho_{\alpha_1}\wedge\cdots\wedge d\rho_{\alpha_{q+q'}}
\\&=\sum_{\alpha_0,\dots,\alpha_q}\rho_{\alpha_0}\xi_{\alpha_0\cdots\alpha_q}\wedge d\rho_{\alpha_1}\wedge\cdots\wedge d\rho_{\alpha_q}\wedge\Di\Ri^{p,q}(\eta)|_{U_{\alpha_q}}
\\&=\Di\Ri^{p,q}(\xi)\wedge\Di\Ri^{p,q}(\eta)
\end{align*}
and (i) holds. 
Property (ii) is deduced just as in \cite[Lemma 8.58]{RoSa1}, \cite[Lemma 5.2.10]{Tark1} or \cite[\S1]{ToTo1}.
\end{proof}
Consider more generally a smooth vector bundle $\Ei$ on $\M$ and the vector space $\check{Z}^q(\GU;\Ai^p\otimes\End\Ei)$ consisting of $(\Ai^p\otimes_{\Ci^\infty_\M}\End\Ei)$-valued $q$-cocycles for any $p,q\geq 0$. 
We introduce a map 
$$
\boldsymbol\Di\boldsymbol\Ri^{p,q}:\check{Z}^q(\GU;\Ai^p\otimes\End\Ei)\to\Ai^{p+q}(\M)\otimes_{C^\infty(\M)}\Gamma^\infty(\M;\End\Ei)
$$
by a formula perfectly analogous to that of $\Di\Ri^{p,q}$,
$$
\boldsymbol\Di\boldsymbol\Ri^{p,q}(\boldsymbol\eta):=\sum_{\alpha_0,\dots,\alpha_q}\rho_{\alpha_0}\boldsymbol\eta_{\alpha_0\cdots\alpha_q}\wedge d\rho_{\alpha_1}\wedge\cdots\wedge d\rho_{\alpha_q}.
$$
Generalizing the cup product \eqref{cupprod}, we define the cup product of two cochains $\boldsymbol\xi\in\check{C}^q(\GU;\Ai^p\otimes\End\Ei)$ and $\boldsymbol\eta\in\check{C}^{q'}(\GU;\Ai^{p'}\otimes\End\Ei)$ in the same way,
\begin{equation}\label{Endcupprodgen}
(\boldsymbol\xi\cup\boldsymbol\eta)_{\alpha_0\cdots\alpha_{q+q'}}
:=(-1)^{pq'}\boldsymbol\xi_{\alpha_0\cdots\alpha_q}\wedge\boldsymbol\eta_{\alpha_q\cdots\alpha_{q+q'}},
\end{equation}
where now the wedge-product notation includes multiplication of the factors in $\End\Ei$. In particular, for $p$ elements $\boldsymbol\eta_1,\dots,\boldsymbol\eta_p$ in $\check{Z}^1(\GU;\Ai^1\otimes\End\Ei)$, we have
\begin{equation}\label{Endcupprod}
(\boldsymbol\eta_1\cup\cdots\cup\boldsymbol\eta_p)_{\alpha_0\cdots\alpha_p}=(-1)^{(p-1)!}(\boldsymbol\eta_1)_{\alpha_0\alpha_1}\wedge(\boldsymbol\eta_2)_{\alpha_1\alpha_2}\wedge\cdots\wedge(\boldsymbol\eta_p)_{\alpha_{p-1}\alpha_p}.
\end{equation}
Letting $\mu:(\Ai^p\otimes\End\Ei)\otimes_\C(\Ai^{p'}\otimes\End\Ei)$ be the multiplication map $\End\Ei\otimes_\C\End\Ei\to\End\Ei$ composed with the antisymmetrization $\Ai^p\otimes_\C\Ai^{p'}\to\Ai^p\otimes_{\Ci^\infty_\M}\Ai^{p'}\to\Ai^{p+p'}$, the product \eqref{Endcupprodgen} is obtained (up to a sign) by applying $\mu$ to the usual cup product with values in the tensor product $(\Ai^p\otimes\End\Ei)\otimes_\C(\Ai^{p'}\otimes\End\Ei)$,
$$
(\boldsymbol\xi\cup\boldsymbol\eta)_{\alpha_0\cdots\alpha_{q+q'}}
=(-1)^{pq'}\mu(\boldsymbol\xi_{\alpha_0\cdots\alpha_q}\otimes_\C\boldsymbol\eta_{\alpha_q\cdots\alpha_{q+q'}}).
$$
This makes it easy to see that $\delta$ is a graded derivation with respect to $\cup$, since this is true for the ordinary cup product and $\delta$ commutes with any morphism of sheaves.

Again the $\boldsymbol\Di\boldsymbol\Ri^{p,q}$'s map cup products of cocycles into wedge products of differential forms by construction, because the scalar-valued 1-forms $d\rho_\beta$ anti-commute with each $\End\Ei$-valued 1-form $(\boldsymbol\eta_s)_{\alpha_{s-1}\alpha_s}$. 


\subsection{Smooth Atiyah $p$-cocycles}
Now let $\Ei$ be a rank-$r$ smooth vector bundle on $\M$ and assume that the finite open covering $\GU=(U_\alpha)_{\alpha=1,\dots,n}$ has been chosen such that $\Ei$ is trivial over each $U_\alpha$. Then $\Ei$ is characterized by its smooth transition functions $g_{\alpha\beta}:U_{\alpha\beta}\to\GeL(r,\C)$. A connection on $\Ei$ can be defined as a collection $\theta=(\theta_\alpha)_{\alpha=1,\dots,n}$, where $\theta_\alpha$ is a 1-form on $U_\alpha$ with values in the algebra $\Mn_r(\C)$ of complex $r\times r$-matrices, such that $\theta_\alpha=dg_{\alpha\beta}\, g_{\alpha\beta}^{-1}+g_{\beta\alpha}^{-1}\theta_\beta g_{\beta\alpha}$ on $U_{\alpha\beta}$. The curvature 2-form $\Theta$ of the connection is then given by 
$$
\Theta=d\theta+\theta\wedge\theta:=
\sum_\alpha\rho_\alpha(d\theta_\alpha+\theta_\alpha\wedge\theta_\alpha)\in\Ai^2(\M)\otimes_{C^\infty(\M)}\Gamma^\infty(\M;\End\Ei).
$$
Here and below we will denote by $d\theta$ the global $\End\Ei$-valued $1$-form $\sum_\alpha\rho_\alpha d\theta_\alpha$, and thus \emph{not} the $\Ai^1\otimes_{\Ci^\infty_\M}\End\Ei$-valued 0-cocycle $d\theta=(d\theta_\alpha)_{\alpha=1,\dots,n}$ even though we use $\theta$ to denote the family $(\theta_\alpha)_{\alpha=1,\dots,n}$. 

For any $p\geq 1$, the $p$th Chern character $\ch_p(\Ei)\in H^{2p}_{\rm DR}(\M;\R)$ of $\Ei$ is by definition the de Rham cohomology class of the closed $2p$-form
\begin{align}
\ch_p(\theta):=\frac{1}{p!(-2\pi i)^p}\tr(\Theta^{\wedge p})=\frac{1}{p!(-2\pi i)^p}\tr\big((d\theta+\theta\wedge\theta)^{\wedge p}\big)\label{chpoftheta}
\end{align}
We are going to focus on the part $\tr\big((d\theta)^{\wedge p}\big)$,
for a special choice of connection $\theta$. First we need some preliminary facts. 

The following lemma is part of an explicit description, which we believe first appeared in \cite[Thm. 1.1.4]{Fedo1}, of Swan's correspondence \cite[\S2]{GVF1}, \cite{Swan1} between smooth vector bundles $\Ei$ on a smooth manifold $\M$ and the projective $C^\infty(\M)$-modules $\Gamma^\infty(\M;\Ei)$ of smooth sections. We can keep it more general by applying the Swan theorem for locally compact spaces \cite[Thm. 8]{Renn1}, in which case the correspondence concerns $\Gamma^\infty_0(\M;\Ei)$, the module of smooth sections vanishing at infinity.
\begin{Lemma}[{\cite[Thm. 1.1.4]{Fedo1}}]\label{Fedosovlemma}
Let $\Ei$ be a rank-$r$ smooth vector bundle over a smooth manifold $\M$, and describe $\Ei$ by transition functions $g_{\alpha\beta}:U_{\alpha\beta}\to\GeL(r,\C)$ relative to some finite open covering $\GU=(U_\alpha)_{\alpha=1,\dots,n}$ of $\M$. Take a partition of unity $(\rho_\alpha)_{\alpha=1,\dots,n}$ of $\M$ subordinate to $\GU$. Define an $nr\times nr$-matrix $I^E$ with entries in $C^\infty(\M)$ by specifying its $(\alpha,\beta)$-block to be
$$
I^E_{\alpha\beta}:=\rho_\alpha g_{\alpha\beta}.
$$
Then $I^E$ is an idempotent such that $I^EC^\infty(\M)^{\oplus nr}\cong\Gamma^\infty_0(\M;\Ei)$ as $C^\infty(\M)$-modules. A different choice of transition functions for (the isomorphism class of) $\Ei$ leads to an idempotent which is stably unitarily equivalent to $I^E$. Conversely, stably unitarily equivalent idempotent matrices with entries in $C^\infty(\M)$ define isomorphic vector bundles. 
\end{Lemma}
It follows that the assignment $\Ei\to I^E$, which clearly takes direct sums to direct sums and tensor products to tensor products, induces the ring isomorphism
$$
K^0(\M)\cong K_0(C_0(\M))
$$
between the topological $K$-theory of the space $\M$ and the topological $K$-theory of the $C^*$-algebra $C_0(\M)$ of continuous functions on $\M$ vanishing at infinity. 

From the identification $\Gamma^\infty_0(\M;\Ei)\cong I^EC^\infty(\M)^{\oplus nr}$ one sees that a smooth vector bundle $\Ei$ always has a connection $\nabla:\Gamma^\infty_0(\M;\Ei)\to\Gamma^\infty_0(\M;\Ei)\otimes_{C^\infty(\M)}\Ai^1(\M)$, namely the  \textbf{Levi-Civita connection} (aka ``Grassmann connection") defined as the compression of the exterior derivative on $C^\infty(\M)^{\oplus nr}$ to the submodule $\Gamma^\infty_0(\M;\Ei)$ \cite[\S8.3]{GVF1}. We denote by $\theta^E$ the collection of local connection 1-forms of the Levi-Civita connection. Its curvature 2-form $\Theta^E$ can be understood as the operator
$$
\Theta^E=I^E\, dI^E\wedge dI^E.
$$
The property $I^EI^E=I^E$ gives  $I^E\, dI^E\wedge dI^E=dI^E\wedge dI^E\, I^E$, which allows $\tr_{\C^{nr}}\big((I^E\,d I^E\wedge dI^E)^{\wedge p}\big)$ to be simplified to $\tr_{\C^{nr}}\big(I^E\, (d I^E\wedge dI^E)^{\wedge p}\big)$. Therefore, if one defines the $p$th Chern character form of $I^E$ to be
$$
\ch_p(I^E):=\frac{1}{p!(-2\pi i)^p}\tr_{\C^{nr}}\big(I^E\, (d I^E\wedge dI^E)^{\wedge p}\big),
$$
then $\ch_p(I^E)$ coincides with the $2p$-form $\ch_p(\theta)$ defined in \eqref{chpoftheta}, for this choice of connection $\theta=\theta^E$, and so $\ch_p(I^E)$ represents the cohomology class $\ch_p(\Ei)$.

\begin{Remark}
Actually, the reference \cite[Thm. 1.1.4]{Fedo1} cited in Lemma \ref{Fedosovlemma} uses the idempotent $\tilde{I}_{\alpha\beta}^E=\sqrt{\rho_\alpha\rho_\beta}g_{\alpha\beta}$, but a direct computation shows that $\tilde{I}^E$ gives the same Chern character forms as $I_{\alpha\beta}^E=\rho_\alpha g_{\alpha\beta}$,
$$
\ch_p(\tilde{I}^E)=\ch_p(I^E),\qquad\forall p\geq 0.
$$
If one makes a polar decomposition $g_{\alpha\beta}=u_{\alpha\beta}(g_{\alpha\beta}^*g_{\alpha\beta})^{1/2}$ (and thereby chooses the Hermitian metric on $\Ei$ corresponding to the standard inner product on $\C^r$ via the trivialization), one can also consider the idempotent
$$
P^E:=\sqrt{\rho_\alpha\rho_\beta}u_{\alpha\beta},
$$ 
which is a moreover a projection, defining the same vector bundle $\Ei$. But whenever $\M$ is an complex-analytic manifold, the $(p',q)$-component of $\ch_p(P^E)$ in $\Ai^{2p}=\bigoplus_{p'+q=2p}\Ai^{p',q}$ is very different from that of $\ch_p(I^E)$. This is easily seen when the $\GeL(r,\C)$-valued transition functions $g_{\alpha\beta}$ are holomorphic, because then the $\Un(r)$-valued transition functions $u_{\alpha\beta}$ are not holomorphic. For our purposes it will be more convenient to use the idempotent $I^E$, even though it is not a projection with respect to the standard inner product on $C^\infty(\M)^{\oplus nr}$.
\end{Remark}

\begin{prop}\label{calcofsmoothChern}
Let $p\geq 1$. In the setting of Lemma \ref{Fedosovlemma}, for the Levi-Civita connection $\theta^E$ one has
\begin{align*}
&\tr_{\C^{nr}}\big((d\theta^E)^{\wedge p}\big)
\\&=(-1)^{(p-1)!}\sum_{\alpha_0,\dots,\alpha_p=1}^n\rho_{\alpha_0}\tr_{\C^r}(d\log g_{\alpha_0\alpha_1}\wedge d\log g_{\alpha_1\alpha_2}\wedge\cdots\wedge d\log g_{\alpha_{p-1}\alpha_p})\wedge d\rho_{\alpha_1}\wedge\cdots\wedge d\rho_{\alpha_p},
\end{align*}
where $d\log g_{\alpha\beta}:=dg_{\alpha\beta}\, g_{\alpha\beta}^{-1}$.
\end{prop}
\begin{proof} 
As an operator on $\Gamma^\infty_0(\M;\Ei)\subset C^\infty(\M)^{\oplus nr}$, the Levi-Civita connection is given by the 1-form $I^E\, dI^E$. Using $\sum_\beta d\rho_\beta=d\big(\sum_\beta\rho_\beta\big)=d\bone=0$, the $(\alpha,\alpha)$-block of $I^E\, dI^E$ becomes $(I^E\, dI^E)_{\alpha\alpha}=-\sum_{\beta}\rho_\alpha\rho_\beta\, dg_{\alpha\beta}\, g_{\beta\alpha}$. The operator $I^E\, dI^E$  defines the local connection 1-forms $\theta_\alpha^E$ via $(I^E\, dI^E)_{\alpha\alpha}=\rho_\alpha\theta_\alpha$, so
\begin{equation}\label{explgrassconn}
\theta_\alpha^E=-\sum_{\beta=1}^n\rho_\beta\, d\log g_{\alpha\beta}\in\Ai^1(U_\alpha)\otimes\Mn_r(\C).
\end{equation}
Applying $\tr_{\C^r}\circ\, d$ and summing over $\alpha$ we have 
\begin{equation}\label{stateforpeq1}
\sum_\alpha\rho_\alpha\tr_{\C^r}(d\theta_\alpha^E)=\sum_{\alpha,\beta}\rho_\alpha\tr_{\C^r}(d\log g_{\alpha\beta})\wedge d\rho_\beta,
\end{equation}
which is the statement for $p=1$. 

For any $p\geq 1$ one could calculate $\tr_{\C^{nr}}\big(I^E\, (d I^E\wedge dI^E)^{\wedge p}\big)$ directly and identify the $\tr_{\C^{nr}}\big((d\theta^E)^{\wedge p}\big)$-part; for $p=2$ for instance, this is the part containing the summand
\begin{align*}
&\tr_{\C^r}(g_{\alpha_0\alpha_1}\, dg_{\alpha_1\alpha_2}\, g_{\alpha_2\alpha_3}\, dg_{\alpha_3\alpha_4}\, g_{\alpha_4\alpha_0})
=\tr_{\C^r}(dg_{\alpha_1\alpha_2}\, g_{\alpha_2\alpha_1}g_{\alpha_1\alpha_3}\, dg_{\alpha_3\alpha_4}\, g_{\alpha_4\alpha_1})
\\&=\tr_{\C^r}\big(dg_{\alpha_1\alpha_2}\, g_{\alpha_2\alpha_1}g_{\alpha_1\alpha_3}\, d(g_{\alpha_3\alpha_1}g_{\alpha_1\alpha_4})\, g_{\alpha_4\alpha_1}\big)
\\&=\tr_{\C^r}(dg_{\alpha_1\alpha_2}\, g_{\alpha_1\alpha_2}^{-1}g_{\alpha_3\alpha_1}^{-1}\, dg_{\alpha_3\alpha_1})+\tr_{\C^r}\big(dg_{\alpha_1\alpha_2}\, g_{\alpha_2\alpha_1}\, dg_{\alpha_1\alpha_4}\, g_{\alpha_4\alpha_1}\big).
\end{align*}
The first term will contribute with zero when we sum over all indices, because we have removed an index $\alpha_4$ which appears as $d\rho_{\alpha_4}$. In the second term we have removed $\alpha_3$ (and $\alpha_0$) but that will just give $\sum_{\alpha_3}\rho_{\alpha_3}=\bone$. So we end up with
\begin{align*}
\tr(d\theta\wedge d\theta)&=\sum_{\alpha_1,\alpha_2,\alpha_4}\rho_{\alpha_1}\tr_{\C^r}\big(dg_{\alpha_1\alpha_2}\, g_{\alpha_2\alpha_1}\wedge d\rho_{\alpha_2}\wedge dg_{\alpha_1\alpha_4}\, g_{\alpha_4\alpha_1}\big)\wedge d\rho_{\alpha_4}
\\&=-\sum_{\alpha_1,\alpha_2,\alpha_4}\rho_{\alpha_1}\tr_{\C^r}\big(dg_{\alpha_1\alpha_2}\, g_{\alpha_2\alpha_1}\, dg_{\alpha_1\alpha_4}\, g_{\alpha_4\alpha_1}\big)\wedge d\rho_{\alpha_2}\wedge d\rho_{\alpha_4}
\end{align*}
which after relabeling gives the stated formula. This computation extends easily to $p\geq 3$ but we can give a cleaner proof which highlights the cup-product structure involved.
For this we shall use the map $\boldsymbol\Di\boldsymbol\Ri^{p,p}$ discussed after Lemma \ref{DRlemma}. 
For later comparison we first rewrite the formula \eqref{explgrassconn} as
$$
\tr(d\theta^E)=\Di\Ri^{1,1}(\xi^1),
$$
where $\xi^1\in\check{Z}^1(\GU;\Ai^1)$ is the 1-cocycle with 
$$
\xi^1_{\alpha\beta}:=\tr_{\C^r}(d\log g_{\alpha\beta}).
$$
Define
$$
\boldsymbol\xi^1_{\alpha\beta}:=d\log g_{\alpha\beta}.
$$
Then $\boldsymbol\xi^1$ satisfies the cocycle condition
\begin{align*}
\boldsymbol\xi_{\alpha\gamma}^1&=dg_{\alpha\gamma}\, g_{\gamma\alpha}
=d(g_{\alpha\beta}g_{\beta\gamma})\, g_{\gamma\beta}g_{\beta\alpha}
=dg_{\alpha\beta}\,g_{\beta\alpha}+g_{\alpha\beta}\, dg_{\beta\gamma}\, g_{\gamma\beta}g_{\beta\alpha}
\\&=\boldsymbol\xi_{\alpha\beta}^1+\Ad(g_{\beta\alpha})(\boldsymbol\xi_{\beta\gamma}^1)
\end{align*}
saying that $\boldsymbol\xi^1$ belongs to $\check{Z}^1(\GU;\Ai^1\otimes\End\Ei)$. 

Now observe that $\boldsymbol\Di\boldsymbol\Ri^{1,1}(\boldsymbol\xi^1)$ coincides with $d\theta^E$. Thus, recalling that $\boldsymbol\Di\boldsymbol\Ri$ maps cup products of cocycles into wedge products, we have
\begin{align*}
\boldsymbol\Di\boldsymbol\Ri^{p,p}((\boldsymbol\xi^1)^{\cup p})&=\boldsymbol\Di\boldsymbol\Ri^{1,1}(\boldsymbol\xi^1)^{\wedge p}
=(d\theta^E)^{\wedge p}.
\end{align*}
Since (cf. \eqref{Endcupprod}) 
\begin{align*}
\tr_{\C^r}\big(\boldsymbol\Di\boldsymbol\Ri^{p,p}((\boldsymbol\xi^1)^{\cup p})\big)&=(-1)^{(p-1)!}\sum_{\alpha_0,\dots,\alpha_p=1}^n\rho_{\alpha_0}\tr_{\C^r}(d\log g_{\alpha_0\alpha_1}\wedge \cdots\wedge d\log g_{\alpha_{p-1}\alpha_p})\wedge d\rho_{\alpha_1}\wedge\cdots\wedge d\rho_{\alpha_p},
\end{align*}
the proof is complete.

\end{proof}
\begin{dfn}\label{smoothAtdef}
The cocycle $\boldsymbol\xi^p:=(\boldsymbol\xi^1)^{\cup p}$, or its scalar-valued version $\xi^p\in\check{Z}^p(\GU;\Ai^p)$ with 
$$
\xi^p_{\alpha_0\cdots\alpha_p}:=(-1)^{(p-1)!}\tr_{\C^r}(d\log g_{\alpha_0\alpha_1}\wedge d\log g_{\alpha_1\alpha_2}\wedge\cdots\wedge d\log g_{\alpha_{p-1}\alpha_p}),
$$
will be called the \textbf{smooth Atiyah $p$-cocycle} of the vector bundle $\Ei$ (with respect to the covering $\GU$).
Sometimes we write $\xi^{E,p}:=\xi^p$ if we want to make explicit reference to the bundle $\Ei$. 
\end{dfn}
The terminology here is by analogy with the well-known Atiyah cocycles of holomorphic vector bundles which we will also need later on, in \S\ref{analytAtiyahsec}.

\subsection{Weil-type expression for $\ch_p:K^0(\M)\to H^{2p}(\M;\Q)$}
Recall that $\check{H}^q(\M;\Fi)=0$ for any $\Ci_\M^\infty$-module $\Fi$ and any $q\geq 1$ \cite[Lemma 4.3.5, Cor. 4.4.5]{Arap2}. For any sheaf $\Fi$ on $\M$ one has $\check{H}^q(\M;\Fi)=0$ if and only if $\check{H}^q(\M;\Fi)=0$ for any open covering $\GU$ of $\M$. 
Denote by $\Zi^k\subset\Ai^k$ the kernel sheaf of the operator $d:\Ai^k\to\Ai^{k+1}$.  From the exact sequence
$$
0\to\C_\M\to\Ci^\infty_\M\overset{d}{\longto}\Ai^1\overset{d}{\longto}\Ai^2\overset{d}{\longto}\cdots,
$$
where all sheaves except the constant sheaf $\C_\M=\Zi^0$ are $\Ci^\infty_\M$-modules, we have \cite[\S IV.6]{Dema1}
\begin{align*}
\check{H}^{2p}(\GU;\C_\M)&\cong\check{H}^{2p-1}(\GU;\Zi^1)\cong\check{H}^{2p-2}(\GU;\Zi^2)\cong\cdots
\\&\cong H^{2p}_{\rm DR}(\GU;\C).
\end{align*}
Suppose now that each intersection $U_{\alpha_0\cdots\alpha_p}\subset\M$ of finitely many elements of the covering $\GU=(U_\alpha)_{\alpha=1,\dots,n}$ is contractible. Then $H^q(\GU;\Fi)=H^q(\M;\Fi)$ for all $\Ci_\M^\infty$-modules $\Fi$, and 
one has
\begin{align}
\check{H^p}(\GU;\Zi^p)&\cong H^{2p}_{\rm DR}(\M;\C)\label{Cechde Rhamiso}.
\end{align}
We now give an explicit description of the isomorphism \eqref{Cechde Rhamiso} which is a direct consequence of \cite[Prop. 9.5]{BoTu1}. See also \cite[\S IV.6, p. 272]{Dema1}.
\begin{Lemma}\label{DRisiso}
Suppose that each intersectionof finitely many elements of the covering $\GU$ is contractible. Then for each $p\in\{1,\dots,p\}$, the map $\Di\Ri^{p,p}:\check{Z}^p(\GU;\Zi^p)\to\Ai^{2p}(\M)$ defined in Lemma \ref{DRlemma} induces a group isomorphism
$$
\check{H^p}(\GU;\Zi^p)\to H^{2p}_{\rm DR}(\M;\C),\qquad [\eta]\to[\Di\Ri^{p,p}(\eta)].
$$
\end{Lemma}
The smooth Atiyah $p$-cocycle $\xi^{E,p}$ (as in Definition \ref{smoothAtdef}) of a smooth vector bundle $\Ei$ takes values in  $\Zi^p$, as a consequence of the property $d(d\log g_{\alpha\beta})=d\log g_{\alpha\beta}\wedge d\log g_{\alpha\beta}$ and the cyclicity of the trace. 

The matrix-valued smooth Atiyah $p$-cocycle $\boldsymbol\xi^{E,p}$ takes values in the $\Ci^\infty_\M$-module $\Ai^p\otimes_{\Ci^\infty_\M}\End\Ei$, but it is not $d$-closed and so it does not define a class in the cohomology group $\check{H}^p(\GU;\Zi^p\otimes_{\Zi^0}\End\Ei)$. Still, we can use the fact that $\boldsymbol\xi^{E,p}$ is a $p$-fold cup product to deduce the following result.

\begin{Lemma}\label{atiyahisotrans}
Let $g_{\alpha\beta}^E$ and $g^F_{\alpha\beta}$ be the transition functions of isomorphic smooth vector bundles $\Ei\cong\Fi$ on $\M$ with respect to some open covering $\GU=(U_\alpha)_{\alpha=1,\dots,n}$ of $\M$. Then the smooth Atiyah $p$-cocycles $\xi^{E,p}$ and $\xi^{F,p}$ belong to the same class in $\check{H}^p(\GU;\Zi^p)$. 
\end{Lemma}
\begin{proof}
Let $g_\alpha^E:\Ei|_{U_\alpha}\to U_\alpha\times\C^r$ be the trivializations of $\Ei$ which relate to the transition functions $g^E_{\alpha\beta}$ via $(g_\alpha^E\circ (g_{\beta}^E)^{-1}(x,v)=(x,g^E_{\alpha\beta}(x)v)$ for $x\in\M$ and $v\in\C^r$.
By assumption there exists a $\Ci^\infty_\M$-module isomorphism $\psi:\Ei\to\Fi$. Recall \cite[Lemma 1.1.3]{Fedo1} that such a $\psi$ gives rise to a collection $(f_\alpha)_{\alpha=1,\dots,n}$ of local matrix-valued functions $f_\alpha\in C^\infty(U_\alpha)\otimes\GeL(r,\C)$, defined by the formula
$$
(x,f_\alpha(x)v)=(g_\alpha^F\circ\psi\circ (g_\alpha^E)^{-1})(x,v),\qquad\forall x\in\M,\, v\in\C^r.
$$
The isomorphism $\psi$ then relates the transition functions via 
$$
g^F_{\alpha\beta}(x)=f_\alpha(x)g^E_{\alpha\beta}(x)f_\beta^{-1}(x).
$$
Therefore,
\begin{align*}
\boldsymbol\xi^{F,1}_{\alpha\beta}&=
\boldsymbol\xi^{E,1}_{\alpha\beta}+d\log f_\alpha-d\log f_\beta.
\end{align*}
Setting $\boldsymbol\zeta_{\alpha\beta}:=d\log f_\alpha-d\log f_\beta$ gives a coboundary in $\check{Z}^1(\GU;\Zi^1\otimes_{\Zi^0}\End\Ei)$. The $d$-closed $1$-forms $\zeta_{\alpha\beta}:=\tr_{\C^r}(\boldsymbol\zeta_{\alpha\beta})$ then define a coboundary in $\check{Z}^1(\GU;\Zi^1)$. Therefore $[\xi^{E,1}]=[\xi^{F,1}]$ in $\check{H}^1(\GU;\Zi^1)$

As remarked after Equation \eqref{Endcupprod}, the \v{C}ech differential $\delta$ is a graded derivation with respect to the cup product defined in \eqref{Endcupprodgen}, implying that the cup product of a cocycle with a coboundary (in either order) is a coboundary. Therefore, for two cocycles $\boldsymbol\xi\in\check{Z}^q(\GU;\Ai^p\otimes\End\Ei)$ and $\boldsymbol\eta\in\check{Z}^{q'}(\GU;\Ai^{p'}\otimes\End\Ei)$ such that $d\tr(\boldsymbol\xi\cup\boldsymbol\eta)=0$, the class $[\tr(\boldsymbol\xi\cup\boldsymbol\eta)]$ in $\check{H}^{q+q'}(\GU;\Zi^{p+p'})$ does not change if we replace $\boldsymbol\xi$ by $\boldsymbol\xi+\boldsymbol\zeta$ for some $d$-closed $q$-coboundary $\boldsymbol\zeta$ (and similarly nothing happens to the class if we add a $d$-closed $q'$-coboundary to $\boldsymbol\eta$). 

Applying the preceding reasoning to $\boldsymbol\xi=\boldsymbol\xi^{E,1}$ and $\boldsymbol\eta=\boldsymbol\xi^{E,p-1}$ gives $[\xi^{E,p}]=[\xi^{F,p}]$ in $\check{H}^p(\GU;\Zi^p)$.
\end{proof}

\begin{thm}\label{Cherrepthm}
For any smooth vector bundle $\Ei$ on $\M$ the closed $2p$-form $\tr\big((d\theta^E)^{\wedge p}\big)$, where $\theta^E$ is the Levi-Civita connection on $\Ei$, represents $p!(-2\pi i)^p$ times the $p$th Chern character of $\Ei$. That is,
\begin{equation}\label{Cherncharexpr}
\frac{1}{p!(-2\pi i)^p}\big[\tr\big((d\theta^E)^{\wedge p}\big)\big]=\ch_p(\Ei)
\end{equation}
in de Rham cohomology.
\end{thm}
\begin{proof} 
Let $\theta=(\theta_\alpha)_{\alpha=1,\dots,n}$ be any connection on $\Ei$, and define
$$
\widetilde{\ch}_p(\theta):=\frac{1}{p!(-2\pi i)^p}\tr\big((d\theta^E)^{\wedge p}\big).
$$
Note first that $[\widetilde{\ch}_p(\theta)]=\ch_p(\Li)$ for any connection $\theta$ on a smooth line bundle $\Li$ over $\M$, because in this case $\theta\wedge\theta=0$. 

We claim that $\widetilde{\ch}_p$ is additive and commutes with pullbacks by smooth maps. Here additivity means that $\widetilde{\ch}_p(\theta)=\widetilde{\ch}_p(\theta^E)+\widetilde{\ch}_p(\theta^F)$ whenever $\theta$ is the connection on a direct sum $\Ei\oplus\Fi$ of smooth vector bundles over $\M$ given by 
$$
\theta:=\theta^E\oplus\theta^F=\begin{pmatrix}
\theta^E&0\\0&\theta^F
\end{pmatrix},
$$
for any connections $\theta^E$ on $\Ei$ and $\theta^F$ on $\Fi$. But 
$$
(d\theta_\alpha)^{\wedge p}=\begin{pmatrix}
(d\theta^E_\alpha)^{\wedge p}&0\\0&(d\theta^F_\alpha)^{\wedge p}
\end{pmatrix},
$$
so taking the trace shows that $\widetilde{\ch}_p(\theta^E\oplus\theta^F)=\widetilde{\ch}_p(\theta^E)+\widetilde{\ch}_p(\theta^F)$, as desired. Next, let $\F$ be a smooth manifold and $\phi:\F\to\M$ a smooth map. Then there is the induced pullback map $\phi^*$ from the $\Ci^\infty_\M$-module $\Ai^{2p}$ to the $\Ci^\infty_\F$-module $\Ai^{2p}$, which we extend to $\End\Ei$-valued forms by applying $\phi^*$ entrywise to local matrix-valued 1-forms. If $\theta$ is any connection on $\Ei$ then $\phi^*\theta$ is a connection on the pullback bundle $\phi^*\Ei$, with respect to the covering of $\F$ be the inverse images $\phi^{-1}(U_\alpha)$. Since $\phi^*$ is multiplicative for the wedge product, we have
$$
\phi^*\tr\big((d\theta)^{\wedge p}\big)=\tr\big((\phi^*(d\theta))^{\wedge p}\big),
$$ 
so that $\phi^*\widetilde{\ch}_p(\theta)=\widetilde{\ch}_p(\phi^*\theta)$. In particular, $\phi^*[\widetilde{\ch}_p(\theta)]:=[\phi^*\widetilde{\ch}_p(\theta)]=[\widetilde{\ch}_p(\phi^*\theta)]$ for the induced map $\phi^*:H^{2p}_{\rm DR}(\M;\C)\to H^{2p}_{\rm DR}(\F;\C)$ on cohomology.

Now take $\theta^E$ to be the Levi-Civita connection on $\Ei$ defined by the idempotent $I^E$ as before.
Let $\F(\Ei)$ be the $r$-flag bundle associated with $\Ei$, and let $\phi:\F(\Ei)\to\M$ be the projection. Then the pullback bundle $\phi^*\Ei$ is isomorphic to a direct sum of smooth line bundles $\Li_1,\dots,\Li_r$ on $\F(\Ei)$ \cite[\S21]{BoTu1},
$$
\phi^*\Ei\cong\Fi:=\bigoplus^r_{j=1}\Li_j.
$$
The Atiyah $p$-cocycle of the vector bundle $\Fi$ has the simple form $\xi^{F,p}=\sum_j\xi^{L_j,p}=\sum_j(\xi^{L_j,1})^{\cup p}$, because $\xi^{L_j,1}=\boldsymbol\xi^{L_j,1}$. By Lemma \ref{atiyahisotrans} we have $[\xi^{\phi^*E,p}]=[\xi^{F,p}]$ and hence $[\Di\Ri^{p,p}(\xi^{\phi^*E,p})]=[\Di\Ri^{p,p}(\xi^{F,p})]$ by Lemma \ref{DRisiso}. Proposition \ref{calcofsmoothChern} gives $p!(-2\pi i)^p\widetilde{\ch}_p(\phi^*\theta^E)=\Di\Ri^{p,p}(\xi^{\phi^*E,p})$. In total, 
\begin{align*}
\phi^*[\widetilde{\ch}_p(\theta^E)]&=[\widetilde{\ch}_p(\phi^*\theta^E)]=\frac{1}{p!(-2\pi i)^p}
[\Di\Ri^{p,p}(\xi^{\phi^*E,p})]=\frac{1}{p!(-2\pi i)^p}[\Di\Ri^{p,p}(\xi^{F,p})]
\\&=\frac{1}{p!(-2\pi i)^p}\sum_{j=1}^r[\Di\Ri^{1,1}(\xi^{L_j,1})^{\wedge p}]=\frac{1}{p!}\sum_{j=1}^r\ch_1(\Li_j)^p
=\ch_p(\Fi)=\ch_p(\phi^*\Ei)
\\&=\phi^*\ch_p(\Ei).
\end{align*}
Since $\phi^*:H^{2p}(\M;\Q)\to H^{2p}(\F(\Ei);\Q)$ is injective, this gives $[\widetilde{\ch}_p(\theta^E)]=\ch_p(\Ei)$ in $H^{2p}(\M;\Q)$.  
\end{proof}

Write $H^{\rm 2*}(\M;\Q):=\bigoplus_{p\geq 0}H^{\rm 2p}(\M;\Q)$. Recall that the Chern character
\begin{equation}\label{cherniso}
\ch_*:K^0(\M)\otimes_\Z\Q\to H^{\rm 2*}(\M;\Q)
\end{equation}
is a ring isomorphism (see 
\cite[\S2, p. 19]{AtHi2}, \cite[Thm. 3]{Karo2}, \cite[Prop. 4.5]{Hatc1} for three different proofs of this fact). 
We will now summarize the above results in a way suitable for our purpose. To simplify the formulation we use the notation $\xi^{E,0}:=1$ where $1=(1_\alpha)_{\alpha=1,\dots,n}$ is the $\Ci^\infty_\M$-valued $0$-cocycle with $1_\alpha\in C^\infty(U_\alpha)$ the restriction to $U_\alpha$ of the identity function $\bone\in C^\infty(\M)$. Thus $\Di^{0,0}(\xi^{E,0}):=\sum_\alpha\rho_\alpha 1_\alpha=\bone$ and $[\Di^{0,0}(\xi^{E,0})]=[\bone]=1\in H^0(\M;\Q)\cong\Q^{b_0(\M)}$. 

\begin{thm}\label{chernisopropan}
Under the inverse of the ring isomorphism 
$$
\bigoplus_{p\geq 0}\check{H}^p(\GU;\Zi^p)\to H^{\rm 2*}_{\rm DR}(\M;\C),\qquad ([\eta^p])_{p\geq 0}\to\sum_{p\geq 0}[\Di\Ri^{p,p}(\eta^p)],
$$
the Chern character $\ch_*(\Ei)$ of a smooth vector bundle $\Ei$ on $\M$ is represented by the normalized sequence
$(1/p!(-2\pi i)^p\xi^p)_{p\geq 0}$ of smooth Atiyah cocycles of $\Ei$, with respect to any choice of covering of $\M$ by open subsets with contractible intersections and any choice of transition functions for the isomorphism class of the bundle $\Ei$.   

Indeed, the ring homomorphism 
\begin{equation}\label{chernisoDR}
\widetilde{\ch}_*:K^0(\M)\otimes_\Z\Q\to H^{\rm 2*}(\M;\Q),\qquad \sum_jq_j[\Ei_j]\to\sum_{p\geq 0}\frac{1}{p!(-2\pi i)^p}\sum_jq_j[\Di\Ri^{p,p}(\xi^{E_j,p})]
\end{equation}
coincides, for any choice of subordinate partition of unity, with the usual Chern character \eqref{cherniso}. Consequently, any cocycle $\eta\in\check{Z}^p(\GU;\Zi^p)$ such that the $2p$-form $(-2\pi i)^{-p}\Di\Ri^{p,p}(\eta)$ has rational periods is cohomologous to a $\Q$-linear combination of Atiyah $p$-cocycles $\xi^{E,p}$ of smooth vector bundle $\Ei$ on $\M$. 
\end{thm}
Note that $p!\ch_p(\Ei)$ is an integral class for each $\Ei$. Hence, if $\eta\in\check{Z}^p(\GU;\Zi^p)$ is such that $(-2\pi i)^{-p}\Di\Ri^{p,p}(\eta)$ has \emph{integral} periods then from Theorem \ref{chernisopropan} we get that $\eta$ is a $\Z$-linear combination of Atiyah cocycles, and by taking direct sums of vector bundles if necessary we see that there are two smooth vector bundles $\Ei$ and $\Fi$ on $\M$ such that $\eta=\xi^{E,p}-\xi^{F,p}$. For odd $p$ we could replace $\Ei$ by $\Ei\oplus\Fi^*$ to obtain $\eta=\xi^{E\oplus F^*,p}$.

Recall that $\check{H}^q(\M;\Fi)=0$ for any $\Ci_\M^\infty$-module $\Fi$ and any $q\geq 1$ \cite[Lemma 4.3.5, Cor. 4.4.5]{Arap2}. For any sheaf $\Fi$ on $\M$ one has $\check{H}^q(\M;\Fi)=0$ if and only if $\check{H}^q(\M;\Fi)=0$ for any open covering $\GU$ of $\M$. 
Denote by $\Zi^k\subset\Ai^k$ the kernel sheaf of the operator $d:\Ai^k\to\Ai^{k+1}$.  From the exact sequence
$$
0\to\C_\M\to\Ci^\infty_\M\overset{d}{\longto}\Ai^1\overset{d}{\longto}\Ai^2\overset{d}{\longto}\cdots,
$$
where all sheaves except the constant sheaf $\C_\M=\Zi^0$ are $\Ci^\infty_\M$-modules, we have \cite[\S IV.6]{Dema1}
\begin{align*}
\check{H}^{2p}(\GU;\C_\M)&\cong\check{H}^{2p-1}(\GU;\Zi^1)\cong\check{H}^{2p-2}(\GU;\Zi^2)\cong\cdots
\\&\cong H^{2p}_{\rm DR}(\GU;\C).
\end{align*}
Suppose now that each intersection $U_{\alpha_0\cdots\alpha_p}\subset\M$ of finitely many elements of the covering $\GU=(U_\alpha)_{\alpha=1,\dots,n}$ is contractible. Then $H^q(\GU;\Fi)=H^q(\M;\Fi)$ for all $\Ci_\M^\infty$-modules $\Fi$, and 
one has
\begin{align}
\check{H^p}(\GU;\Zi^p)&\cong H^{2p}_{\rm DR}(\M;\C)\label{Cechde Rhamiso}.
\end{align}
We now give an explicit description of the isomorphism \eqref{Cechde Rhamiso} which is a direct consequence of \cite[Prop. 9.5]{BoTu1}. See also \cite[\S IV.6, p. 272]{Dema1}.
\begin{Lemma}\label{DRisiso}
Suppose that each intersection of finitely many elements of the covering $\GU$ is contractible. Then for each $p\in\{1,\dots,p\}$, the map $\Di\Ri^{p,p}:\check{Z}^p(\GU;\Zi^p)\to\Ai^{2p}(\M)$ defined in Lemma \ref{DRlemma} induces a group isomorphism
$$
\check{H^p}(\GU;\Zi^p)\to H^{2p}_{\rm DR}(\M;\C),\qquad [\eta]\to[\Di\Ri^{p,p}(\eta)].
$$
\end{Lemma}
\begin{Remark}[Assumptions on $\M$]
Lemma \ref{DRisiso} is the first statement in this paper where $\M$ cannot be arbitrary, because it is assumed to possess a \emph{finite} open covering with special properties (so that every smooth vector bundle can be trivialized by the same covering). The assumption is satisfied by any compact $\M$. 
\end{Remark}
The smooth Atiyah $p$-cocycle $\xi^{E,p}$ (as in Definition \ref{smoothAtdef}) of a smooth vector bundle $\Ei$ takes values in  $\Zi^p$, as a consequence of the property $d(d\log g_{\alpha\beta})=d\log g_{\alpha\beta}\wedge d\log g_{\alpha\beta}$ and the cyclicity of the trace. 

The matrix-valued smooth Atiyah $p$-cocycle $\boldsymbol\xi^{E,p}$ takes values in the $\Ci^\infty_\M$-module $\Ai^p\otimes_{\Ci^\infty_\M}\End\Ei$, but it is not $d$-closed and so it does not define a class in the cohomology group $\check{H}^p(\GU;\Zi^p\otimes_{\Zi^0}\End\Ei)$. Still, we can use the fact that $\boldsymbol\xi^{E,p}$ is a $p$-fold cup product to deduce the following result.

\begin{Lemma}\label{atiyahisotrans}
Let $g_{\alpha\beta}^E$ and $g^F_{\alpha\beta}$ be the transition functions of isomorphic smooth vector bundles $\Ei\cong\Fi$ on $\M$ with respect to some open covering $\GU=(U_\alpha)_{\alpha=1,\dots,n}$ of $\M$. Then the smooth Atiyah $p$-cocycles $\xi^{E,p}$ and $\xi^{F,p}$ belong to the same class in $\check{H}^p(\GU;\Zi^p)$. 
\end{Lemma}
\begin{proof}
Let $g_\alpha^E:\Ei|_{U_\alpha}\to U_\alpha\times\C^r$ be the trivializations of $\Ei$ which relate to the transition functions $g^E_{\alpha\beta}$ via $(g_\alpha^E\circ (g_{\beta}^E)^{-1}(x,v)=(x,g^E_{\alpha\beta}(x)v)$ for $x\in\M$ and $v\in\C^r$.
By assumption there exists a $\Ci^\infty_\M$-module isomorphism $\psi:\Ei\to\Fi$. Recall \cite[Lemma 1.1.3]{Fedo1} that such a $\psi$ gives rise to a collection $(f_\alpha)_{\alpha=1,\dots,n}$ of local matrix-valued functions $f_\alpha\in C^\infty(U_\alpha)\otimes\GeL(r,\C)$, defined by the formula
$$
(x,f_\alpha(x)v)=(g_\alpha^F\circ\psi\circ (g_\alpha^E)^{-1})(x,v),\qquad\forall x\in\M,\, v\in\C^r.
$$
The isomorphism $\psi$ then relates the transition functions via 
$$
g^F_{\alpha\beta}(x)=f_\alpha(x)g^E_{\alpha\beta}(x)f_\beta^{-1}(x).
$$
We have
\begin{align*}
\boldsymbol\xi^{F,1}_{\alpha\beta}-\boldsymbol\xi^{E,1}_{\alpha\beta}&=\Ad(g_{\alpha\beta}^F)(d\log f_\beta)-\Ad(g_{\alpha\beta}^E)(d\log f_\alpha).
\end{align*}
From now on we make the transition functions $\Ad(g_{\alpha\beta}^E)$ and $\Ad(g_{\alpha\beta}^F)$ for $\End\Ei\cong\End\Fi$ implicit, since they will disappear after tracing. Thus,
\begin{align*}
\boldsymbol\xi^{F,1}_{\alpha\beta}&=
\boldsymbol\xi^{E,1}_{\alpha\beta}+d\log f_\alpha-d\log f_\beta.
\end{align*}
Setting $\boldsymbol\zeta_{\alpha\beta}:=d\log f_\alpha-d\log f_\beta$ gives a coboundary in $\check{Z}^1(\GU;\Zi^1\otimes_{\Zi^0}\End\Ei)$. The $d$-closed $1$-forms $\zeta_{\alpha\beta}:=\tr_{\C^r}(\boldsymbol\zeta_{\alpha\beta})$ then define a coboundary in $\check{Z}^1(\GU;\Zi^1)$. Therefore $[\xi^{E,1}]=[\xi^{F,1}]$ in $\check{H}^1(\GU;\Zi^1)$

As remarked after Equation \eqref{Endcupprod}, the \v{C}ech differential $\delta$ is a graded derivation with respect to the cup product defined in \eqref{Endcupprodgen}, implying that the cup product of a cocycle with a coboundary (in either order) is a coboundary. Therefore, for two cocycles $\boldsymbol\xi\in\check{Z}^q(\GU;\Ai^p\otimes\End\Ei)$ and $\boldsymbol\eta\in\check{Z}^{q'}(\GU;\Ai^{p'}\otimes\End\Ei)$ such that $d\tr(\boldsymbol\xi\cup\boldsymbol\eta)=0$, the class $[\tr(\boldsymbol\xi\cup\boldsymbol\eta)]$ in $\check{H}^{q+q'}(\GU;\Zi^{p+p'})$ does not change if we replace $\boldsymbol\xi$ by $\boldsymbol\xi+\boldsymbol\zeta$ for some $d$-closed $q$-coboundary $\boldsymbol\zeta$ (and similarly nothing happens to the class if we add a $d$-closed $q'$-coboundary to $\boldsymbol\eta$). 

Applying the preceding reasoning to $\boldsymbol\xi=\boldsymbol\xi^{E,1}$ and $\boldsymbol\eta=\boldsymbol\xi^{E,p-1}$ gives $[\xi^{E,p}]=[\xi^{F,p}]$ in $\check{H}^p(\GU;\Zi^p)$.
\end{proof}

\begin{thm}\label{Cherrepthm}
For any smooth vector bundle $\Ei$ on $\M$ the closed $2p$-form $\tr\big((d\theta^E)^{\wedge p}\big)$, where $\theta^E$ is the Levi-Civita connection on $\Ei$, represents $p!(-2\pi i)^p$ times the $p$th Chern character of $\Ei$. That is,
\begin{equation}\label{Cherncharexpr}
\frac{1}{p!(-2\pi i)^p}\big[\tr\big((d\theta^E)^{\wedge p}\big)\big]=\ch_p(\Ei)
\end{equation}
in de Rham cohomology.
\end{thm}
\begin{proof} 
Let $\theta=(\theta_\alpha)_{\alpha=1,\dots,n}$ be any connection on $\Ei$, and define
$$
\widetilde{\ch}_p(\theta):=\frac{1}{p!(-2\pi i)^p}\tr\big((d\theta^E)^{\wedge p}\big).
$$
Note first that $[\widetilde{\ch}_p(\theta)]=\ch_p(\Li)$ for any connection $\theta$ on a smooth line bundle $\Li$ over $\M$, because in this case $\theta\wedge\theta=0$. 

We claim that $\widetilde{\ch}_p$ is additive and commutes with pullbacks by smooth maps. Here additivity means that $\widetilde{\ch}_p(\theta)=\widetilde{\ch}_p(\theta^E)+\widetilde{\ch}_p(\theta^F)$ whenever $\theta$ is the connection on a direct sum $\Ei\oplus\Fi$ of smooth vector bundles over $\M$ given by 
$$
\theta:=\theta^E\oplus\theta^F=\begin{pmatrix}
\theta^E&0\\0&\theta^F
\end{pmatrix},
$$
for any connections $\theta^E$ on $\Ei$ and $\theta^F$ on $\Fi$. But 
$$
(d\theta_\alpha)^{\wedge p}=\begin{pmatrix}
(d\theta^E_\alpha)^{\wedge p}&0\\0&(d\theta^F_\alpha)^{\wedge p}
\end{pmatrix},
$$
so taking the trace shows that $\widetilde{\ch}_p(\theta^E\oplus\theta^F)=\widetilde{\ch}_p(\theta^E)+\widetilde{\ch}_p(\theta^F)$, as desired. Next, let $\F$ be a smooth manifold and $\phi:\F\to\M$ a smooth map. Then there is the induced pullback map $\phi^*$ from the $\Ci^\infty_\M$-module $\Ai^{2p}$ to the $\Ci^\infty_\F$-module $\Ai^{2p}$, which we extend to $\End\Ei$-valued forms by applying $\phi^*$ entrywise to local matrix-valued 1-forms. If $\theta$ is any connection on $\Ei$ then $\phi^*\theta$ is a connection on the pullback bundle $\phi^*\Ei$, with respect to the covering of $\F$ be the inverse images $\phi^{-1}(U_\alpha)$. Since $\phi^*$ is multiplicative for the wedge product, we have
$$
\phi^*\tr\big((d\theta)^{\wedge p}\big)=\tr\big((\phi^*(d\theta))^{\wedge p}\big),
$$ 
so that $\phi^*\widetilde{\ch}_p(\theta)=\widetilde{\ch}_p(\phi^*\theta)$. In particular, $\phi^*[\widetilde{\ch}_p(\theta)]:=[\phi^*\widetilde{\ch}_p(\theta)]=[\widetilde{\ch}_p(\phi^*\theta)]$ for the induced map $\phi^*:H^{2p}_{\rm DR}(\M;\C)\to H^{2p}_{\rm DR}(\F;\C)$ on cohomology.

Now take $\theta^E$ to be the Levi-Civita connection on $\Ei$ defined by the idempotent $I^E$ as before.
Let $\F(\Ei)$ be the $r$-flag bundle associated with $\Ei$, and let $\phi:\F(\Ei)\to\M$ be the projection. Then the pullback bundle $\phi^*\Ei$ is isomorphic to a direct sum of smooth line bundles $\Li_1,\dots,\Li_r$ on $\F(\Ei)$ \cite[\S21]{BoTu1},
$$
\phi^*\Ei\cong\Fi:=\bigoplus^r_{j=1}\Li_j.
$$
The Atiyah $p$-cocycle of the vector bundle $\Fi$ has the simple form $\xi^{F,p}=\sum_j\xi^{L_j,p}=\sum_j(\xi^{L_j,1})^{\cup p}$, because $\xi^{L_j,1}=\boldsymbol\xi^{L_j,1}$. By Lemma \ref{atiyahisotrans} we have $[\xi^{\phi^*E,p}]=[\xi^{F,p}]$ and hence $[\Di\Ri^{p,p}(\xi^{\phi^*E,p})]=[\Di\Ri^{p,p}(\xi^{F,p})]$ by Lemma \ref{DRisiso}. Proposition \ref{calcofsmoothChern} gives $p!(-2\pi i)^p\widetilde{\ch}_p(\phi^*\theta^E)=\Di\Ri^{p,p}(\xi^{\phi^*E,p})$. In total, 
\begin{align*}
\phi^*[\widetilde{\ch}_p(\theta^E)]&=[\widetilde{\ch}_p(\phi^*\theta^E)]=\frac{1}{p!(-2\pi i)^p}
[\Di\Ri^{p,p}(\xi^{\phi^*E,p})]=\frac{1}{p!(-2\pi i)^p}[\Di\Ri^{p,p}(\xi^{F,p})]
\\&=\frac{1}{p!(-2\pi i)^p}\sum_{j=1}^r[\Di\Ri^{1,1}(\xi^{L_j,1})^{\wedge p}]=\frac{1}{p!}\sum_{j=1}^r\ch_1(\Li_j)^p
=\ch_p(\Fi)=\ch_p(\phi^*\Ei)
\\&=\phi^*\ch_p(\Ei).
\end{align*}
Since $\phi^*:H^{2p}(\M;\Q)\to H^{2p}(\F(\Ei);\Q)$ is injective, this gives $[\widetilde{\ch}_p(\theta^E)]=\ch_p(\Ei)$ in $H^{2p}(\M;\Q)$.  
\end{proof}

Write $H^{\rm 2*}(\M;\Q):=\bigoplus_{p\geq 0}H^{\rm 2p}(\M;\Q)$. Recall that the Chern character
\begin{equation}\label{cherniso}
\ch_*:K^0(\M)\otimes_\Z\Q\to H^{\rm 2*}(\M;\Q)
\end{equation}
is a ring isomorphism (see 
\cite[\S2, p. 19]{AtHi2}, \cite[Thm. 3]{Karo2}, \cite[Prop. 4.5]{Hatc1} for three different proofs of this fact). 
We will now summarize the above results in a way suitable for our purpose. To simplify the formulation we use the notation $\xi^{E,0}:=1$ where $1=(1_\alpha)_{\alpha=1,\dots,n}$ is the $\Ci^\infty_\M$-valued $0$-cocycle with $1_\alpha\in C^\infty(U_\alpha)$ the restriction to $U_\alpha$ of the identity function $\bone\in C^\infty(\M)$. Thus $\Di^{0,0}(\xi^{E,0}):=\sum_\alpha\rho_\alpha 1_\alpha=\bone$ and $[\Di^{0,0}(\xi^{E,0})]=[\bone]=1\in H^0(\M;\Q)\cong\Q^{b_0(\M)}$. 

In the following theorem we will again need the assumption on $\M$ that there is a covering $\GU=(U_\alpha)_{\alpha=1,\dots,n}$ which is both finite and has contractible finite intersections among its elements. 
\begin{thm}\label{chernisopropan}
Under the inverse of the ring isomorphism 
$$
\bigoplus_{p\geq 0}\check{H}^p(\GU;\Zi^p)\to H^{\rm 2*}_{\rm DR}(\M;\C),\qquad ([\eta^p])_{p\geq 0}\to\sum_{p\geq 0}[\Di\Ri^{p,p}(\eta^p)],
$$
the Chern character $\ch_*(\Ei)$ of a smooth vector bundle $\Ei$ on $\M$ is represented by the normalized sequence
$(1/p!(-2\pi i)^p\xi^p)_{p\geq 0}$ of smooth Atiyah cocycles of $\Ei$, with respect to any choice of covering of $\M$ by open subsets with contractible intersections and any choice of transition functions for the isomorphism class of the bundle $\Ei$.   

Indeed, the ring homomorphism 
\begin{equation}\label{chernisoDR}
\widetilde{\ch}_*:K^0(\M)\otimes_\Z\Q\to H^{\rm 2*}(\M;\Q),\qquad \sum_jq_j[\Ei_j]\to\sum_{p\geq 0}\frac{1}{p!(-2\pi i)^p}\sum_jq_j[\Di\Ri^{p,p}(\xi^{E_j,p})]
\end{equation}
coincides, for any choice of subordinate partition of unity, with the usual Chern character \eqref{cherniso}. Consequently, any cocycle $\eta\in\check{Z}^p(\GU;\Zi^p)$ such that the $2p$-form $(-2\pi i)^{-p}\Di\Ri^{p,p}(\eta)$ has rational periods is cohomologous to a $\Q$-linear combination of Atiyah $p$-cocycles $\xi^{E,p}$ of smooth vector bundle $\Ei$ on $\M$. 
\end{thm}
Note that $p!\ch_p(\Ei)$ is an integral class for each $\Ei$. Hence, if $\eta\in\check{Z}^p(\GU;\Zi^p)$ is such that $(-2\pi i)^{-p}\Di\Ri^{p,p}(\eta)$ has \emph{integral} periods then from Theorem \ref{chernisopropan} we get that $\eta$ is a $\Z$-linear combination of Atiyah cocycles, and by taking direct sums of vector bundles if necessary we see that there are two smooth vector bundles $\Ei$ and $\Fi$ on $\M$ such that $\eta=\xi^{E,p}-\xi^{F,p}$. For odd $p$ we could replace $\Ei$ by $\Ei\oplus\Fi^*$ to obtain $\eta=\xi^{E\oplus F^*,p}$.

\subsection{Exact forms versus coboundaries}

In this section we show that one can express every global exact $2p$-form on $\M$ as $\Di\Ri^{p,p}(\xi^{E,p})$ for some trivial smooth vector bundle $\Ei$. 

Throughout, assume that that every intersection of finitely many open set in the covering $\GU=(U_\alpha)_{\alpha=1,\dots,n}$ is contractible. We shall need a theorem of Narasimhan--Ramanan \cite{NaRa1} in the following version.  
\begin{Lemma}\label{strongNaRalemma}
Let $\Ei$ be a smooth vector bundle over $\M$. Then every connection $\theta$ on $\Ei$ is gauge equivalent to the Levi-Civita connection
$$
\theta^F_\alpha=\sum_\beta\rho_\beta\, d\log g^F_{\alpha\beta}
$$ 
of some smooth vector bundle $\Fi$ isomorphic to $\Ei$. More precisely, if $g^F_{\alpha\beta}=f_\alpha g_{\alpha\beta}^Ef_\beta^{-1}$ and $\zeta_\alpha:=df_\alpha\, f_\alpha^{-1}$, then 
$$
\theta_\alpha=\theta^F_\alpha+\zeta_\alpha.
$$
The same partition of unity $(\rho_\alpha)_{\alpha=1,\dots,n}$ works for all connections.
\end{Lemma}
\begin{proof} 
Let $\Pis$ be the principal $\GeL(r,\C)$-bundle of frames for $\Ei$. Let $\vartheta_\alpha$ be the connection on the trivial bundle $U_\alpha\times\GeL(r,\C)$ obtained from the Maurer--Cartan 1-form $dg\, g^{-1}$ on $\GeL(r,\C)$.
Given any connection $\theta$ on $\Ei$, corresponding to a connection $\tilde{\theta}\in\Ai^1(\Pis)\otimes\Gg\Gl(r,\C)$ on $\Pis$, Narasimhan--Ramanan show \cite{NaRa1} that the local 1-form $\tilde{\theta}_\alpha$ on $\Pis|_{U_\alpha}$ is obtained by pulling back $\vartheta_\alpha$ via some trivialization $\Psi_\alpha:\Pis|_{U_\alpha}\to U_\alpha\times G$,
$$
\tilde{\theta}_\alpha=\Psi_\alpha^*\vartheta_\alpha.
$$
The zero section $C^\infty(\M)\hookrightarrow C^\infty(\Pis)$ maps the identity in $C^\infty(\M)$ to the identity on $C^\infty(\Pis)$. Therefore, the patching of the $\tilde{\theta}_\alpha$'s into $\theta_\alpha$ can be done with the same partition $(\rho_\alpha)_{\alpha=1,\dots,n}$ that we used for the cover $\GU$ of the base manifold (more precisely we use the image of $\rho_\alpha$ under the embedding, and denote it by the same symbol),
$$
\tilde{\theta}=\sum_\alpha\rho_\alpha\tilde{\theta}_\alpha=\sum_\alpha\rho_\alpha\Psi_\alpha^*\vartheta_\alpha.
$$
The need of special choice of trivializations $\Psi_\alpha$ requires only a change in isomorphism class of $\Pis$, so the vector bundle $\Fi$ associated with the principal $\GeL(r,\C)$-bundle defined by the $\Psi_\alpha$'s and the standard $\GeL(r,\C)$-representation is isomorphic to the original bundle $\Ei$. Letting $s_\alpha$ be the local frame for $\Fi$ defined by $\Psi_\alpha(s_\alpha(x))=(x,\bone_r)$, we obtain the connection on $\Fi|_{U_\alpha}$ defined by $\tilde{\theta}$ by pulling back va $s_\alpha$,
$$
s_\alpha^*\tilde{\theta}=\sum_\beta\rho_\beta s_\alpha^*\Psi_\beta^*\vartheta_\beta=\sum_\beta\rho_\beta\, d\log g_{\alpha\beta}^F=\theta_\alpha^F.
$$
The given connection $\theta$ is by definition the pullback of $\tilde{\theta}$ by a local frame for $\Pis$. Since $\Ei\cong\Fi$, this gives the relation $\theta_\alpha=\theta_\alpha^F+f^{-1}_\alpha\, df_\alpha$ for any gauge transformation $(f_\alpha)_{\alpha=1,\dots,n}\in\check{C}^0(\GU;\GeL(r,\Ci^\infty_\M))$ defining an isomorphism between $\Ei$ and $\Fi$. 
\end{proof}

We have already mentioned the paper \cite{Weil1} in which Weil constructed the \v{C}ech--de Rham isomorphism. In the same paper he proved what is now sometimes called the ``geometric quantization theorem" (due to its independent proof by Kostant motivated by quantization): every $\Z(1)$-valued closed 2-form $\Theta\in\Ai^2(\M)$ on a smooth manifold $\M$ is the curvature 2-form of some connection on some smooth line bundle over $\M$ \cite[Thm. 2.2.15]{Bryl3}. We will prove a strengthening of this theorem which also works for $2p$-forms for any $p$. 

We define $\Di\Ri$ as before using a fixed partition of unity $(\rho_\alpha)_{\alpha=1,\dots,n}$ subordinate to $\GU$.
\begin{prop}\label{exacttwoformsprop}
Every globally $d$-closed 2-form on $\M$ with $\Z(1)$-valued periods is equal to the Levi-Civita curvature $d\theta^L=\Di\Ri^{1,1}(\xi^{1,L})$ of some smooth line bundle on $\M$. Every globally $d$-exact 2-form on $\M$ is equal to the Levi-Civita curvature of some trivial line bundle $\Li\cong\Ci^\infty_\M$. 
\end{prop}
\begin{proof}
Suppose that $\Theta\in\Ai^2(\M)$ is a $d$-closed 2-form. Then we can construct a cocycle $\xi\in\check{Z}^1(\GU;\Zi^1)$ such that $[\Di\Ri^{1,1}(\eta)]=[\Theta]$ in $H^2_{\rm DR}(\M;\C)$. Indeed (cf. \cite[\S IV.6.5]{Dema1}), by the Poincaré lemma there are $1$-forms $\theta_\alpha\in\Ai^1(U_\alpha)$ such that $\Theta|_{U_\alpha}=d\theta_\alpha$, and we define $\xi_{\alpha\beta}:=\theta_\beta-\theta_\alpha$. 
For $q=0$, Equation \eqref{dofDReq} for the 0-cochain $(\theta_\alpha)_{\alpha=1,\dots,n}$ modifies to
$$
d\Di\Ri^{0,1}(\theta)=\Di\Ri^{1,1}(\delta\theta)+\Di\Ri^{0,1}(d\theta),
$$
i.e.
\begin{equation}\label{Thetaversusxi}
d\Di\Ri^{0,1}(\theta)=\Di\Ri^{1,1}(\xi)+\Theta. 
\end{equation}
Every $\Z(1)$-valued 1-cocycle $\xi\in\check{Z}^1(\GU;\Zi^1)$ is of the form $\xi_{\alpha\beta}=d\log g_{\alpha\beta}$ for some $\GeL(1,\Ci^\infty_\M)$-valued 1-cocycle $(g_{\alpha\beta})_{\alpha,\beta=1,\dots,n}$, because the Poincaré lemma applies to each intersection $U_{\alpha\beta}$. Since the $\Z(1)$-valued cocycles generate $\check{Z}^1(\GU;\Zi^1)$ over $\C$, every element of $\check{Z}^1(\GU;\Zi^1)$ is a $\C$-linear combination of 1-cocycles like $d\log g_{\alpha\beta}$.

Assume now that $\Theta$ has $\Z(1)$-valued periods, so that the corresponding cocycles are $\Z(1)$-valued, and fix such a $\xi\in\check{Z}^0(\GU;\Zi^1)$. Thus $\xi=\xi^{L,1}$ is the smooth Atiyah 1-cocycle of some smooth line bundle $\Li$ on $\M$, and since the corresponding potential $\theta$ transform as a connection on $\Li$, its exterior derivative $\Theta$ is the curvature of a connection on $\Li$. In fact, $\Theta$ is the curvature of a connection on \emph{every} line bundle in the smooth isomorphism class of $\Li$, since we can add any $d$-exact $0$-cocycle $\zeta_\alpha=d\log f_\alpha$ to $\theta_\alpha$ without changing $\Theta$ and $\theta+\zeta$ is a connection on the line bundle defined by the transition functions $f_\alpha g_{\alpha\beta}f_\beta^{-1}$. After fixing a potential $\theta$ for $\Theta$, the connection on each line bundle $\Li$ with curvature $\Theta$ is unique, given by $\theta+\zeta^L$.

By Lemma \ref{strongNaRalemma}, the potential $\theta$ is gauge equivalent to the Levi-Civita connection on some line bundle in the same isomorphism class. So $\theta_\alpha+\zeta^L_\alpha=\theta^L_\alpha$ for some $\Li$ and $d\theta^L_\alpha=d\theta_\alpha=\Theta|_{U_\alpha}$ as desired.

If $\Theta\in d\Ai^1(\M)$ is a globally exact then we can choose a potential $\theta'$ with $\delta\theta'=0$. Since any two potentials differ by some $\zeta^L$, we have $\theta'=-\theta^L+\zeta^L$ for some smooth line bundle $\Li$. So $\theta^L$ differs by a gauge transformation from a global 1-form, implying that $\Li$ is trivial.

\end{proof}
The following remark is not essential but merely gives some extra observations related to the above proposition.
\begin{Remark}
Let $h:\check{C}^q(\GU;\Ai^p)\to\check{C}^{q-1}(\GU;\Ai^p)$ be the \v{C}ech homotopy operator defined by the partition of unity, i.e. 
$$
(h\eta)_{\alpha_1\cdots\alpha_p}:=\sum_{\alpha_0}\rho_{\alpha_0}\eta_{\alpha_0\cdots\alpha_p}.
$$
The equation $\id=h\delta+\delta h$ gives a decomposition
$$
\check{C}^1(\GU;\Ai^1)=\delta h\check{C}^1(\GU;\Ai^1)\oplus h\delta\check{C}^1(\GU;\Ai^1)=\check{Z}^1(\GU;\Ai^1)\oplus h\delta\check{C}^1(\GU;\Ai^1).
$$ 
Observe that the subspace $\check{C}^1(\GU;\Zi^1)$ can be described as
$$
\check{C}^1(\GU;\Zi^1)=\{\zeta=\theta'+\theta\in\check{Z}^1(\GU;\Ai^1)\oplus h\delta\check{C}^1(\GU;\Ai^1)|\ d\theta'=-d\theta\}.
$$ 
What Lemma \ref{strongNaRalemma} says is that we have an isomorphism
$$
\delta h:\check{C}^1(\GU;\Zi^1)/\check{Z}^1(\GU;\Zi^1)\to \check{Z}^1(\GU;\Ai^1)/\check{Z}^1(\GU;\Zi^1),\qquad \zeta\to\theta':=\delta h\zeta.
$$
In other words, for any $d$-exact 2-form $\Theta\in d\Ai^1(\M)$ we can always find local potentials $\theta_\alpha$ such that $h\theta=0$. 

We see that the Levi-Civita connections $\theta=h\delta\zeta$ of flat bundles, and their global versions $\theta'=\delta h\zeta$, are very natural objects: they are the components of $d$-closed cochains $\zeta=d\log f\in\check{C}^0(\GU;\Zi^1)$ in the decomposition $\check{C}^0(\GU;\Ai^1)=\Ran(\delta^{(-1)})\oplus\Ran(\delta^{(0)})$.  

A trival line bundle $\Li\cong\Ci^\infty_\M$ is flat (i.e. has constant transition functions $g_{\alpha\beta}=f_\alpha f_\beta^{-1}$) if and only if $\delta\zeta^L=0$, where $\zeta^L_\alpha=d\log f_\alpha$ as before. 
\end{Remark}

Since every $d$-closed form on $U_\alpha$ is $d$-exact, hence logarithmically $d$-exact, every element $\zeta$ of $\check{C}^0(\GU;\Zi^1)$ is of the form $\zeta_\alpha=d\log f_\alpha$ for some invertible functions $f_\alpha\in C^\infty(U_\alpha)$. So every coboundary $\delta\zeta\in\check{Z}^1(\GU;\Zi^1)$ is of the form $(\delta\zeta)_{\alpha\beta}=d\log f_\alpha-d\log f_\beta$, and hence every $d$-exact 2-form can be expressed as
\begin{align*}
\Di\Ri^{1,1}(\delta\zeta)&=\sum_{\alpha,\beta}\rho_\alpha(\delta\zeta)_{\alpha\beta}\wedge d\rho_\beta=\sum_{\alpha,\beta}\rho_\alpha\, d\log(f_\alpha f_\beta^{-1})\wedge d\rho_\beta
\\&=\sum_{\alpha,\beta}\rho_\alpha\, d\log g_{\alpha\beta}\wedge d\rho_\beta
\end{align*}
if we define $g_{\alpha\beta}:=f_\alpha f_\beta^{-1}$. 


\begin{Lemma}[{\cite[Thm. 4.1]{Datta1}, \cite[Prop. 1]{PiTa1}}]\label{Dattalemma}
Every exact differential $2p$-form on $\M$ can be expressed as
$$
\ch_p(\theta)=\sum_{j=1}^r(d\theta^{j})^{\wedge p}\in d\Ai^{2p-1}(\M)
$$
for some connections $\theta^j$ on the trivial smooth line bundle $\Ci^\infty_\M$ over $\M$, where
$$
\theta:=\theta^1\oplus\cdots\oplus\theta^r=
\begin{pmatrix}
\theta^1&0&\cdots&0\\
0&\theta^2&\cdots&0\\
\vdots&0&\ddots&\vdots\\
0&\cdots&0&\theta^r
\end{pmatrix}\in\Ai^1(\M)\otimes\Mn_r(\C)
$$
is the direct-sum connection on $(\Ci^\infty_\M)^{\oplus r}$. 
\end{Lemma}
Observe that the connections in the preceeding lemma are arbitrary globally defined 1-forms $\theta^j\in\Ai^1(\M)$. We will work with line bundles which are merely \emph{isomorphic} to the trivial line bundle $\Ci^\infty_\M$; a connection $\theta$ on such a line bundle is not a globally defined 1-form on $\M$ but its exterior derivative $d\theta$ coincides with the exterior derivative of some global 1-form.

\begin{thm}\label{exactcoboundthm}
A global $2p$-form on $\M$ is $d$-exact if and only if it equals $\Di\Ri^{p,p}(\xi^{E,p})$ for a smooth vector bundle $\Ei$ in the isomorphism class of the trivial bundle $(\Ci^\infty_\M)^{\oplus r}$ (of some rank $r\geq 1$). 
\end{thm}
\begin{proof}
Combining Lemma \ref{Dattalemma} with Proposition \ref{exacttwoformsprop} we see that every global $d$-exact 2-form on $\M$ can be expressed as
\begin{align*}
\ch_p(\theta^E)&=\sum_{j=1}^r(d\theta^{L_j})^{\wedge p}=\sum_{j=1}^r\Di\Ri^{1,1}(\xi^{L_j,1})^{\wedge p}
\\&=\Di\Ri^{p,p}(\xi^{E,p})
\end{align*}
where $\theta^E=\theta^{L_1}\oplus\cdots\oplus\theta^{L_r}$ is the Levi-Civita connection on some trivial bundle $\Ei=\Li_1\oplus\cdots\oplus\Li_r$.
\end{proof}


\begin{cor}\label{exactcoboundcor}
Every closed $2p$-form on $\M$ with $\Z(p)$-periods can be expressed as $\Di\Ri^{p,p}(\xi^{E,p})$ for some smooth vector bundle $\Ei$.
\end{cor}
\begin{proof}
This follows from Theorem \ref{exactcoboundthm} and the fact that the Chern character is an isomorphism. 
\end{proof}




\section{$(p,p)$-classes}

\subsection{Atiyah $(p,0)$-cocycles}\label{analytAtiyahsec}
Let $\M$ be a compact complex-analytic manifold. Suppose that $\Ei$ is a rank-$r$ holomorphic vector bundle on $\M$ defined by holomorphic transition functions $g_{\alpha\beta}:U_{\alpha\beta}\to\GeL(r,\C)$ with respect to some $\Ei$-trivializing open covering $\GU=(U_\alpha)_{\alpha=1,\dots,n}$ of $\M$. We know from \cite[Prop. 12]{Atiy1} (see also \cite{ABST1, Kapr3, KRR1}) that the obstruction in $\check{H}^1(\GU;\Omega^1\otimes\End\Ei)$ to the existence of a global holomorphic connection on $\Ei$ is represented by the 1-cocycle $\boldsymbol\xi^{1,0}$, with values in $\Omega^1\otimes_{\Oi_\M}\End\Ei$, defined by
$$
\boldsymbol\xi^{1,0}_{\alpha\beta}
:=\pd\log g_{\alpha\beta},
$$
where $\pd\log g_{\alpha\beta}:=\pd g_{\alpha\beta}\, g^{-1}_{\alpha\beta}$. Note that $\bar{\pd}g_{\alpha\beta}=0$ since $\Ei$ is holomorphic, so that $\boldsymbol\xi^{1,0}=\boldsymbol\xi^1$ is exactly the smooth Atiyah $p$-cocycle of $\Ei$ in the sense of Definition \ref{smoothAtdef}. 
More generally, the smooth Atiyah $p$-cocycle for any $p\geq 1$ simplifies, for a holomorphic vector bundle $\Ei$, to
\begin{align*}
\boldsymbol\xi_{\alpha_0\cdots\alpha_p}^p=\boldsymbol\xi^{p,0}_{\alpha_0\cdots\alpha_p}
:=(-1)^{(p-1)!}\pd\log g_{\alpha_0\alpha_1}\wedge\cdots\wedge\pd\log g_{\alpha_{p-1}\alpha_p}.
\end{align*}
Now if $\Ei$ is any smooth vector bundle on $\M$ with smooth transition functions $g_{\alpha\beta}$, we can still define the cocycles
\begin{equation}\label{Atiyahpcocycle}
\boldsymbol\xi^{p,0}\in\check{Z}^p(\GU;\Ai^{p,0}\otimes_{\Ci^\infty_\M}\End\Ei),\qquad \boldsymbol\xi^{p,0}_{\alpha_0\cdots\alpha_p}:=(-1)^{(p-1)!}\pd\log g_{\alpha_0\alpha_1}\wedge\cdots\wedge\pd\log g_{\alpha_{p-1}\alpha_p},
\end{equation}
\begin{equation}\label{tracedAtiyahpcocycle}
\xi^{p,0}\in\check{Z}^p(\GU;\Ai^{p,0}),\qquad \xi^{p,0}_{\alpha_0\cdots\alpha_p}:=\tr_{\C^r}(\boldsymbol\xi^{p,0}_{\alpha_0\cdots\alpha_p}),
\end{equation}
In contrast to the standard setting where the $g_{\alpha\beta}$'s are holomorphic, it rarely happens that $\bar{\pd}\xi^{p,0}=0$
(let alone $\bar{\pd}\boldsymbol\xi^{p,0}=0$). We refer to $\xi^{p,0}$ or $\boldsymbol\xi^{p,0}$ as the \textbf{Atiyah} $(p,0)$-\textbf{cocycle} of $\Ei$ (with respect to the covering $\GU$). Note that $\boldsymbol\xi^{p,0}$ is the $p$-fold cup product of $\boldsymbol\xi^{1,0}$ with itself,
$$
\boldsymbol\xi^{p,0}=(\boldsymbol\xi^{1,0})^{\cup p},
$$
if we use the cup product defined in Equation \eqref{Endcupprod}. 


\subsection{Harvey's Dolbeault isomorphism}\label{Harveysec}
\begin{Lemma}[{\cite{Harv1}, \cite[Thm. 5.2.12]{Tark1}, \cite[\S1]{ToTo1}}]\label{HarvDolbisolemma}
Let $\M$ be a compact complex-analytic manifold and let $\GU=(U_\alpha)_{\alpha=1,\dots,n}$ be a Stein covering of $\M$ (see e.g. \cite[\S 1.4.4]{GrRe1}). Then for any choice of partition of unity $(\rho_\alpha)_{\alpha=1,\dots,n}$ subordinate to $\GU$, the maps
$$
\Di^{p,q}:\check{C}^q(\GU;\Ai^{p,0})\to\Ai^{p,q}(\M)
$$ 
defined by
$$
\Di^{p,q}(\eta):=\sum_{\alpha_0,\dots,\alpha_q}\rho_{\alpha_0}\eta_{\alpha_0\cdots\alpha_q}\wedge \bar{\pd}\rho_{\alpha_1}\wedge\cdots\wedge\bar{\pd}\rho_{\alpha_q}
$$
induce the Dolbeault isomorphisms 
$$
\check{H}^q(\GU;\Omega^p)\cong H^{p,q}_{\bar{\pd}}(\M).
$$
\end{Lemma}

We will write $\Di:=\Di^{p,p}$ when $p\in\{1,\dots,\dim_\C\M\}$ is already given. 
\begin{Lemma}\label{HarvAtiyahlemma}
In the setting of Lemma \ref{HarvDolbisolemma}, let $\Ei$ be a holomorphic vector bundle on $\M$ with holomorphic transition functions $g_{\alpha\beta}:U_{\alpha\beta}\to\GeL(r,\C)$ and let $\xi^{p,0}$ be the Atiyah $(p,0)$-cocycle of $\Ei$, as in \eqref{tracedAtiyahpcocycle}. Suppose that $\M$ is Kähler, so that we can identify the Dolbeault cohomology group $H_{\bar{\pd}}^{p,p}(\M)$ with a subspace $H^{p,p}(\M)\subset H^{2p}_{\rm DR}(\M;\C)$. Then the $(p,p)$-form
$$
\Di(\xi^{p,0})=(-1)^{(p-1)!}\sum_{\alpha_0,\dots,\alpha_p}\rho_{\alpha_0}\tr_{\C^r}(\pd\log g_{\alpha_0\alpha_1}\wedge\cdots\wedge\pd\log g_{\alpha_{p-1}\alpha_p})\wedge \bar{\pd}\rho_{\alpha_1}\wedge\cdots\wedge\bar{\pd}\rho_{\alpha_p}
$$
is a Dolbeault representative of $p!(-2\pi i)^p$ times the $p$th Chern character $\ch_p(\Ei)\in H^{p,p}(\M)\cap H^{2p}(\M;\Q)$.
\end{Lemma}
\begin{proof}
Let $I^E$ be the idempotent $C^\infty(\M)$-valued matrix with blocks $I_{\alpha\beta}^E=\rho_\alpha g_{\alpha\beta}$ as in Lemma \ref{Fedosovlemma}. Write the Levi-Civita connection $\theta=\theta^E$ as $\theta_\alpha=\theta_\alpha^{1,0}+\theta_\alpha^{0,1}$, with $\theta^{1,0}_\alpha\in\Ai^1(U_\alpha)\otimes\Mn_r(\C)$ and $\theta^{0,1}_\alpha\in\Ai^1(U_\alpha)\otimes\Mn_r(\C)$. Since $\bar{\pd}g_{\alpha\beta}=0$, we see from Equation \eqref{explgrassconn} that $\theta^{0,1}=0$. Hence, the curvature 2-form decomposes as $\Theta=\Theta^{1,1}+\Theta^{2,0}$, with 
$$
\Theta^{1,1}=\bar{\pd}\theta^{1,0}=\tr_{\C^{nr}}(I^E\, \pd I^E\wedge\bar{\pd}I^E)\in\Gamma^\infty(\M;\Ai^{1,1}\otimes_{\Ci^\infty_\M}\End\Ei),
$$
$$
\Theta^{2,0}=\pd\theta^{1,0}+\theta^{1,0}\wedge\theta^{1,0}\in\Gamma^\infty(\M;\Ai^{2,0}\otimes_{\Ci^\infty_\M}\End\Ei).
$$
In particular, $\tr\big((\bar{\pd}\theta^{1,0})^{\wedge p}\big)=\tr_{\C^{nr}}((\Theta^{1,1})^{\wedge p})$. Using $\bar{\pd}g_{\alpha\beta}=0$ again, the $(p,p)$-component of the formula in Proposition \ref{calcofsmoothChern} reads
$$
\tr_{\C^{nr}}((\Theta^{1,1})^{\wedge p})=\Di(\xi^{p,0}).
$$
This says precisely that $\Di(\xi^{p,0})$ is the ``Atiyah $p$-form", in the sense of \cite[Def. 1.8]{ABST1}, of the Levi-Civita connection on $\Ei$. Therefore, $\Di(\xi^{p,0})$ is a Dolbeault representative of $p!(-2\pi i)^p$ times the $p$th Chern character of $\Ei$ by \cite[\S\S1.2, 1.3]{ABST1}. To recall how this goes, let $\theta^{E,h}$ be the Chern connection for some Hermitian metric $h$ on $\Ei$. Then $\ch_p(\theta^{E,h})$ is of type $(p,p)$ and is both a Dolbeault and a de Rham representative of $\ch_p(\Ei)$. The $(p,p-1)$-component $\cs_{p,p-1}(\theta^E,\theta^{E,h})$ of the Chern--Simons transgression from $\ch_p(I^E)$ to $\ch_p(\theta^{E,h})$ then satisfies
$$
\bar{\pd}\cs_{p,p-1}(\theta^E,\theta^{E,h})=\frac{1}{p!(-2\pi i)^p}\Di(\xi^{p,0})-\ch_p(\theta^{E,h}),
$$  
whence the class of $\Di(\xi^{p,0})/p!(-2\pi i)^p$ in $H_{\bar{\pd}}^{p,p}(\M)$ is equal to that of $\ch_p(\theta^{E,h})$ and identifies with $\ch_p(\Ei)\in H^{p,p}(\M)\subset H^{2p}_{\rm DR}(\M;\C)$. 
\end{proof}

\subsection{From Dolbeault to de Rham cohomology}



Let $\Ei$ be a smooth vector bundle on a compact complex-analytic manifold $\M$. The Atiyah $(p,0)$-cocycle $\xi^{p,0}=\xi^{E,p,0}$ of $\Ei$ is $\pd$-closed. If $\xi^{p,0}$ is moreover $\bar{\pd}$-closed, then it defines a de Rham cohomology class $(-2\pi i)^{-p}[\Di\Ri^{p,p}(\xi^{p,0})]\in H^{2p}(\M;\Q)$ with rational periods. We want to understand which analytic cocycles $\eta\in\check{Z}^p(\GU;\Omega^p)$ are cohomologous to linear combinations of $\bar{\pd}$-closed Atiyah $(p,0)$-cocycles. For that we want to find the obstruction to mapping the class $[\eta]$ into de Rham cohomology where we can apply Theorem \ref{chernisopropan}. Since $\bar{\pd}\eta=0$ does not imply $d\eta=0$ for arbitrary $\eta\in\check{Z}^p(\GU;\Omega^p)$, the differential $2p$-form $\Di\Ri^{p,p}(\eta)$ need not be $d$-closed. In essence, there is no easy way of relating Dolbeault cohomology classes to de Rham cohomology classes. 

One solution could be to try to use not $\check{Z}^p(\GU;\Omega^p)$ but a space of cocycles which calculates Bott--Chern cohomology (see \cite[\S3.8]{Kooi1}), provided one could find an analogue of the map $\Di$ defined in Lemma \ref{HarvAtiyahlemma} which induces an isomorphism of some \v{C}ech cohomology with Bott--Chern cohomology. This becomes complicated but the problem resolves completely if we restrict attention to compact Kähler manifolds, due to the following fact.
\begin{Lemma}[{\cite[Prop. 8.9]{Dema2}}]\label{allholoareharm}
If $\M$ is a compact Kähler manifold then 
$$
\{\omega\in\Omega^p|\ \pd\omega=0\}=\Omega^p,
$$
i.e. every holomorphic $p$-form on $\M$ is $\pd$-closed. 
\end{Lemma}
Thus, while for arbitrary $(p,q)$-classes a Dolbeault representative is not necessarily a de Rham representative even for a compact Kähler manifold, $(p,0)$-forms on a compact Kähler manifold are $d$-closed whenever they are $\bar{\pd}$-closed. Moreover, it is easy to see that if $\M$ is compact Kähler then the isomorphism $H^{p,p}_{\bar{\pd}}(\M)\to H^{p,p}(\M)\subset H^{2p}_{\rm DR}(\M;\C)$ is explicitly given by
$$
H^{p,p}_{\bar{\pd}}(\M)\to H^{p,p}(\M),\qquad [\Di^{p,p}(\eta)]\to[\Di\Ri^{p,p}(\eta)],
$$
because $[\eta]=[\eta']$ in $\check{H}^p(\M;\Omega^p)$ if and only if $[\Di\Ri^{p,p}(\eta)]=[\Di\Ri^{p,p}(\eta')]$. 
With this in mind we can deduce the following, which is the main result of the paper and may be useful in an attempt to prove the Hodge conjecture. Recall our notation: in order to suppress factors of $(-2\pi i)^{-p}$ we consider in $\C$ the subgroup $\Z(p):=(2\pi i)^p\Z$ and the $\Q$-vector subspace $\Q(p):=(2\pi i)^p\Q$ generated by $(-2\pi i)^p=(2\pi i)^p(-1)^p$. 
\begin{thm}[Hodge representatives]\label{steponethm}
Let $\M$ be a compact Kähler manifold. Then every class in $H^{p,p}(\M)\cap H^{2p}(\M;\Q(p))$ has a Dolbeault representative which is a $\Q$-linear combination of the $(p,p)$-forms $\Di(\xi^{p,0})$ where the $\xi^{p,0}$'s are $\bar{\pd}$-closed Atiyah $(p,0)$-cocycles of smooth vector bundles over $\M$.
\end{thm}
\begin{proof}
Let $H^{2p}_0(\M;\Z(p))$ be the torsion-free part of $H^{2p}(\M;\Z(p))$, regarded as a subgroup of $H^{2p}(\M;\C)$. By Lemma \ref{HarvDolbisolemma}, every class in $H^{p,p}(\M)$, and in particular every class in $H^{p,p}(\M)\cap H^{2p}_0(\M;\Z(p))$, has a representative of the form $\Di^{p,p}(\eta)$ for some $\eta\in\check{Z}^p(\M;\Omega^p)$. Observe that $d\eta=0$ by Lemma \ref{allholoareharm}, which implies that $\Di\Ri^{p,p}(\eta)$ is a closed $2p$-form on $\M$ representing a $\Z(p)$-valued de Rham class. Now Theorem \ref{chernisopropan} shows that a de Rham $2p$-class is $\Z(p)$-valued if and only if it can be represented a $\Z$-linear combination of $2p$-forms of the type $\Di\Ri^{p,p}(\xi^{E,p})$ for some smooth vector bundle $\Ei$. Since we will anyway take $\Q$-linear combinations in the end, let us consider only $\eta$'s for which we can find $\Ei$ such that 
$$
[\Di\Ri^{p,p}(\xi^{E,p})]=[\Di\Ri^{p,p}(\eta)]\in H^{2p}_{\rm DR}(\M;\R),
$$
or equivalently 
$$
[\xi^{E,p}]=[\eta]\in\check{H}^p(\M;\Zi^p).
$$
Since $\eta$ is of type $(p,0)$, this gives $\eta=\xi^{E,p}+\delta\zeta=\xi^{E,p,0}+(\delta\zeta)^{p,0}$ for the Atiyah $(p,0)$-cocycle $\xi^{E,p,0}$ and some $\zeta\in\check{Z}^{p-1}(\GU;\Zi^p)$, where $\zeta^{p,0}$ denotes the $\Ai^{p,0}$-valued part of $\zeta$. But $d\zeta=0$ does not imply $d\zeta^{p,0}=0$ (hence not $\bar{\pd}\zeta^{p,0}=0$). Therefore, it can well happen that $\bar{\pd}\xi^{E,p,0}\ne 0$.

To our rescue comes Theorem \ref{exactcoboundthm} which shows that $\delta\zeta$ is the smooth Atiyah $p$-cocycle of some smoothly trivial vector bundle on $\M$. Therefore, $\eta=\xi^{E,p}+\delta\zeta$ is the smooth Atiyah $p$-cocycle of some smooth vector bundle $\Fi$ on $\M$ with $[\Fi]-[\Ei]=0\in K^0(\M)\otimes_\Z\Q$. So a $\bar{\pd}$-closed Atiyah $(p,0)$-cocycle $\eta=\xi^{F,p}=\xi^{F,p,0}$ represents the given $\Z(p)$-valued $(p,p)$-class. 

The space $H^{p,p}(\M)\cap H^{2p}(\M;\Q(p))$ of $\Q(p)$-valued $(p,p)$-classes is generated over $\Q$ by $H^{p,p}(\M)\cap H^{2p}_0(\M;\Z(p))$, so every $\Q(p)$-valued $(p,p)$-class can be represented by a $\Q$-linear combination of $\bar{\pd}$-closed Atiyah $(p,0)$-cocycles.
\end{proof}



Recall that if $U$ is an open subset of any complex-analytic manifold $\M$, a smooth $(p,0)$-form $\eta\in\Ai^{p,0}(U)$ is holomorphic, i.e. belongs to $\bwedge^p\Omega^1(U)$, if and only if $\bar{\pd}\eta=0$ \cite[Thm. II.2.6]{GrRe2}. That is, the $\Oi_\M$-module $\Omega^p:=\Ker(\bar{\pd}:\Ai^{p,0}\to\Ai^{p,1})$ can be described as the sheaf of holomorphic $p$-forms,
$$
\Omega^p=\bwedge^p\Omega^1.
$$
Therefore, if $\xi^{p,0}=\tr_{\C^r}(\pd\log g_{\alpha_0\alpha_1}\wedge\cdots\wedge\pd\log g_{\alpha_{p-1}\alpha_p})$ is a $\bar{\pd}$-closed Atiyah cocycle of some vector bundle $\Ei$ on a compact complex-analytic manifold $\M$, we can write $\xi^{p,0}$ as a $\C$-linear combination of wedge products of holomorphic 1-forms. This does not imply that the 1-forms $\pd\log g_{\alpha\beta}$ are holomorphic. To show that this is always the case (up to smooth isomorphism of bundles) when $\M$ is projective (and $\xi^{p,0}\ne 0$) is to prove the Hodge $(p,p)$-conjeccture.

\subsection{Algebraic classes}\label{algsection}
We can make Theorem \ref{steponethm} a bit stronger by specializing to projective Kähler manifolds. 

\subsubsection{Preliminaries on Koszul cocycles}
Let us first fix the notation and recall some facts about algebraic \v{C}ech cocycles which might be useful to keep in mind for those that are interested in applying the results of the present paper. 

Projective space $\C\Pb^{n-1}$ has a covering by Stein manifolds $V_\alpha\cong\C^{n-1}$, which are the projectivizations of the open subsets $\{(z_1,\dots,z_n)\in\C^n|\ z_\alpha\ne 0\}$ of $\C^n$. 
When we are discussing a (smooth) projective variety $\M\subset\C\Pb^{n-1}$, we will always consider the covering $\GU=(U_\alpha)_{\alpha=1,\dots,n}$ given by
$$
U_\alpha:=\M\cap V_\alpha=\M\cap\{(z_1,\dots,z_n)\in\C^n|\ z_\alpha\ne 0\}/\C^\times.
$$
Let $A=\bigoplus_{m\in\N_0}A_m$ be the homogeneous coordinate ring of $\M$ with respect to the given embedding $\M\subset\C\Pb^{n-1}$, where $A_m$ is the subspace of elements which are homogeneous of degree $m$. Let $Z_\alpha$ be the normalized coordinate function $z_\alpha/\|z\|$ on $\C^n$ restricted the affine cone $\Spec(A)\subset\C^n$ of $\M$. Then we have
$$
\sum^n_{\alpha=1}Z_\alpha Z_\alpha^*=\bone\in C^\infty(\M),
$$
and $Z_1,\dots,Z_n$ generate $A$ as an algebra over $\C$. 

For a graded $A$-module $F=\bigoplus_{k\in\Z}F_k$ we denote by $F_k$ the subspace of elements which are homogeneous of degree $k$. 
Write $F_{(Z_\alpha)}$ for the homogeneous localization of $F$ at the multiplicatively closed subset $\{Z_\alpha^m|\ m\in\N_0\}$, i.e. $F_{(Z_\alpha)}:=(F_{Z_\alpha})_0$ is the degree-0 part of the usual module of fractions $F_{Z_\alpha}=A_{Z_\alpha}\otimes_AF$. Recall \cite[Thm. 13.20]{GoWe1}
that every quasi-coherent algebraic sheaf $\Fi$ on $\M$ is of the form 
$$
\Fi(U_\alpha)=F_{(Z_\alpha)},\qquad\forall \alpha\in\{1,\dots,n\}
$$
for some graded $A$-module $F$. Therefore the \v{C}ech cochain groups $\check{C}^q(\GU;\Fi):=\bigoplus_{\alpha_0,\dots,\alpha_q}\Fi(U_{\alpha_0\cdots\alpha_q})$ of the algebraic sheaf $\Fi$ are given by
$$
\check{C}^q(\GU;\Fi)=\bigoplus_{\alpha_0,\dots,\alpha_q}F_{(Z_{\alpha_0}\cdots Z_{\alpha_q})}.
$$
So an $\Fi$-valued \v{C}ech $q$-cochain is a family $\eta$ of elements $\eta_{\alpha_0\cdots\alpha_q}\in F_{(Z_{\alpha_0}\cdots Z_{\alpha_q})}$ of the form
$$
\eta_{\alpha_0\cdots\alpha_q}=\frac{\eta_{\alpha_0\cdots\alpha_q}^{(m)}}{(Z_{\alpha_0}\cdots Z_{\alpha_q})^m}
$$
for some $m\in\N$ (since the number of sets in the open covering is finite, we can find $m$ large enough to make it the same for all $\alpha_0,\dots,\alpha_q\in\{1,\dots,n\}$). The \v{C}ech coboundary of $\eta$ then takes the form
\begin{align*}
(\delta \eta)_{\alpha_0\cdots\alpha_{q+1}}&=\frac{\eta_{\alpha_1\cdots\alpha_{q+1}}^{(m)}}{(Z_{\alpha_1}\cdots Z_{\alpha_{q+1}})^m}-\frac{\eta_{\alpha_0\alpha_2\cdots\alpha_{q+1}}^{(m)}}{(Z_{\alpha_0}Z_{\alpha_2}\cdots Z_{\alpha_{q+1}})^m}+\cdots+(-1)^{q+1}\frac{\eta^{(m)}_{\alpha_0\cdots\alpha_q}}{(Z_{\alpha_0}\cdots Z_{\alpha_q})^m}
\\&=\frac{Z_{\alpha_0}^m\eta_{\alpha_1\cdots\alpha_{q+1}}^{(m)}-Z_{\alpha_1}^m\eta_{\alpha_0\alpha_2\cdots\alpha_{q+1}}^{(m)}+\cdots+(-1)^{q+1}Z_{\alpha_{q+1}}^m\eta_{\alpha_0\cdots\alpha_q}^{(m)}}{(Z_{\alpha_0}\cdots Z_{\alpha_{q+1}})^m}.
\end{align*}
The condition $\delta\eta=0$ says that the above expression is zero in the localization $F_{(Z_{\alpha_0}\cdots Z_{\alpha_{q+1}})}$, which is the same as saying that there is an $r\in\N_0$ such that
$$
(Z_{\alpha_0}\cdots Z_{\alpha_{q+1}})^r(Z_{\alpha_0}^m\eta_{\alpha_1\cdots\alpha_{q+1}}^{(m)}-Z_{\alpha_1}^m\eta_{\alpha_0\alpha_2\cdots\alpha_{q+1}}^{(m)}+\cdots+(-1)^{q+1}Z_{\alpha_{q+1}}\eta_{\alpha_0\cdots\alpha_q}^{(m)})=0.
$$
So if we (following \cite[\S VII.5]{MuOd1}) set $l:=m+r$ and
\begin{equation}\label{indsystKosz}
\eta_{\alpha_0\cdots\alpha_q}^{(l)}:=(Z_{\alpha_0}\cdots Z_{\alpha_q})^{l-m}\eta_{\alpha_0\cdots\alpha_q}^{(m)}
\end{equation}
so that $\eta_{\alpha_0\cdots\alpha_q}=\eta_{\alpha_0\cdots\alpha_q}^{(l)}/(Z_{\alpha_0}\cdots Z_{\alpha_q})^l$, then the cocycle condition $(\delta\eta)_{\alpha_0\cdots\alpha_{q+1}}=0$ on $U_{\alpha_0\cdots\alpha_q}$ is equivalent to 
\begin{equation}\label{algtrasfo}
Z_{\alpha_0}^l\eta_{\alpha_1\cdots\alpha_{q+1}}^{(l)}-Z_{\alpha_1}^l\eta_{\alpha_0\alpha_2\cdots\alpha_{q+1}}^{(l)}
+\cdots+(-1)^{q+1}Z_{\alpha_{q+1}}^l\eta_{\alpha_0\cdots\alpha_q}^{(l)}=0.
\end{equation}
in $F_{(Z_{\alpha_0}\cdots Z_{\alpha_{q+1}})}$. In other words, provided that $\eta_{\alpha_0\cdots\alpha_q}$ is antisymmetric under interchange of any two indices, $\eta^{(m)}$ is a $(q+1)$-cocycle in the homogeneous Koszul complex $(K^\bullet(Z_1^m,\dots,Z_n^m;F)_0$ of the $A$-module $F$ (the degree-0 part of the usual Koszul complex as defined e.g. in \cite[\S17]{Eise2}). The subcomplex $\check{C}^\bullet_{\rm alt}(\GU;\Fi)\subset\check{C}^\bullet(\GU;\Fi)$ of antisymmetric \v{C}ech $q$-cochains is thus an inductive limit
$$
\check{C}^q_{\rm alt}(\GU;\Fi)=\lim_{m\to\infty}K^{q+1}(Z_1^m,\dots,Z_n^m;F)_0
$$
of homogeneous Koszul complexes, where the inductive system is given by $\eta^{(m)}\to\eta^{(l)}$ as in Equation \eqref{indsystKosz}; see \cite[\S\S61-64]{Serre2} (and also \cite[\S5]{BrSh1}, \cite[\S2.3]{BrFo1}, \cite{Green1}, \cite[\S5]{Ilye1}, \cite[\S VII.5]{MuOd1}). 



\subsubsection{Algebraic representatives of Hodge classes}

As above, let $A$ be the homogeneous coordinate ring of the variety under consideration. We denote by $\Omega_A^1$ the graded $A$-module of Kähler differentials and for $p\geq 2$ we set $\Omega^p_A:=\bwedge^p\Omega^1_A$.
\begin{thm}[Hodge representatives for projective manifolds]\label{projsteponethm}
Let $\M\subset\C\Pb^{n-1}$ be a smooth projective variety. Then every class in $H^{p,p}(\M)\cap H^{2p}(\M;\Q(p))$ has a Dolbeault representative of the form $\Di^{p,p}(\eta)$ where $\eta$ is a $\Q$-linear combination of $\bar{\pd}$-closed Atiyah $(p,0)$-cocycles of smooth vector bundles over $\M$ and moreover we can choose the cocycle $\eta$ to be algebraic, i.e. with
$$
\eta_{\alpha_0\cdots\alpha_p}=\frac{\eta_{\alpha_0\cdots\alpha_q}^{(m)}}{(Z_{\alpha_0}\cdots Z_{\alpha_q})^m}\in (\Omega_A^p)_{(Z_{\alpha_0}\cdots Z_{\alpha_p})}
$$
and transforming according to \eqref{algtrasfo}. 
\end{thm}
\begin{proof}
Recall that the sheaf $\Omega^p_{\rm alg}$ of algebraic $p$-forms on $\M$ can be defined for all $\alpha_0,\dots,\alpha_p\in\{1,\dots,n\}$ by
$$
\Omega^p_{\rm alg}(U_{\alpha_0\cdots\alpha_p}):=(\Omega_A^p)_{(Z_{\alpha_0}\cdots Z_{\alpha_p})}.
$$
Since algebraic $p$-forms are in particular holomorphic, we have an inclusion \cite[\S11]{Serre1}
$$
\check{C}^p(\GU;\Omega^p_{\rm alg})\subset\check{C}^p(\GU;\Omega^p)
$$
which intertwines the \v{C}ech differentials. By GAGA \cite[\S12, Thm. 1]{Serre1},
$$
\check{H}^p(\GU;\Omega^p_{\rm an})\cong \check{H}^p(\GU;\Omega^p_{\rm alg}).
$$
Therefore, Lemma \ref{HarvDolbisolemma} gives that every class in $H^{p,p}(\M)\cap H^{2p}(\M;\Q(p))$ has a representative of the form $\Di^{p,p}(\eta)$ for some $\eta\in\check{Z}^p(\GU;\Omega^p_{\rm alg})$. We have $d\eta=0$ by Lemma \ref{allholoareharm}, so $\Di\Ri^{p,p}(\eta)$ is a closed $2p$-form on $\M$ representing a $\Q(p)$-valued DeRham class. Hence $\eta$ is \v{C}ech cohomologous to a rational linear combination of Atiyah $(p,0)$-cocycles (see the proof of Theorem \ref{steponethm}). 
\end{proof}
Every holomorphic vector bundle on a projective variety is isomorphic to a vector bundle with algebraic transition functions \cite[Prop. 18]{Serre2}. It follows that the Atiyah $(p,0)$-cocycle of a holomorphic vector bundles is cohomologous to an algebraic Atiyah $(p,0)$-cocycle. 
Theorem \ref{projsteponethm} says that a $\bar{\pd}$-closed Atiyah $(p,0)$-cocycle $\xi^{E,p,0}$ is always cohomologous to an algebraic one, even if the vector bundle is not holomorphic. Perhaps one could use this result to show that such a vector bundle $\Ei$ has holomorphic transition functions, up to smooth isomorphism, what would validate the Hodge conjecture.


\end{document}